\documentclass[a4paper,10pt, reqno]{amsart}

\usepackage[main=english]{babel}

\usepackage{amsmath,amssymb,latexsym, amsfonts}
\usepackage{verbatim}
\usepackage{newcent}
\usepackage{mathbbol}
\usepackage[colorlinks=true, allcolors=black]{hyperref}

\usepackage{mathabx}
\usepackage{bbm}
\usepackage{pifont}
\usepackage{manfnt}

\usepackage[dvipsnames]{xcolor}

\usepackage{tikz-cd}
\usepackage{mathtools}

\usepackage{upgreek}  

\usepackage{tikz}
\usetikzlibrary{backgrounds}

\usepackage[all,ps]{xy}
\usepackage{color}

\usepackage{stmaryrd}

\usepackage{capt-of}

\makeatletter
\@namedef{subjclassname@2020}{%
  \textup{2020} Mathematics Subject Classification}
\makeatother

\usepackage[OT2, T1]{fontenc} 

\usepackage[textsize=scriptsize]{todonotes}

\usepackage{scalerel,stackengine}  

\newcommand\pig[1]{\scalerel*[5pt]{\big#1}{%
  \ensurestackMath{\addstackgap[1.5pt]{\big#1}}}}

\newcommand{\sss}{\scriptscriptstyle}
\newcommand{\foo}{\scriptstyle}

\newcommand{\compactlist}[1]{\setlength{\itemsep}{0pt} \setlength{\parskip}{0pt} \setlength{\leftskip}{-0.#1em}}

\newcommand{\Sum}{\textstyle\sum\limits}

\numberwithin{equation}{section}

\DeclareRobustCommand{\SkipTocEntry}[5]{}

\theoremstyle{plain}

\newtheorem{theorem}{Theorem}[section]

\newtheorem*{theorem*}{Theorem}

\newtheorem{prop}[theorem]{Proposition}
\newtheorem{lemma}[theorem]{Lemma}
\newtheorem{lem}[theorem]{Lemma}

\newtheorem{cor}[theorem]{Corollary}

\theoremstyle{definition}
\newtheorem{definition}[theorem]{Definition}

\newtheorem{rem}[theorem]{Remark}

\newcommand{\gd}{\delta} 
\newcommand{\gD}{\Delta} 
\newcommand{\gve}{\varepsilon} 
  
\newcommand{\gvf}{\varphi}

\newcommand{\gL}{\Lambda}

\newcommand{\gs}{\sigma}

\newcommand{\cM}{{\mathcal M}}

\newcommand{\cO}{{\mathcal O}}
\newcommand{\cP}{{\mathcal P}}

\newcommand{\Hom}{\operatorname{Hom}}

\newcommand{\Tor}{\operatorname{Tor}}
\newcommand{\Ext}{\operatorname{Ext}}
\newcommand{\Coext}{\operatorname{Coext}}
\newcommand{\Cotor}{\operatorname{Cotor}}

\newcommand{\bfTor}{\operatorname{{\mathbf{Tor}}}}
\newcommand{\bfExt}{\operatorname{{\mathbf{Ext}}}}
\newcommand{\bfCotor}{\operatorname{{\mathbf{Cotor}}}}
\newcommand{\bfCoext}{\operatorname{{\mathbf{Coext}}}}

\newcommand{\id}{{\rm id}}

\newcommand{\pl}{\partial}

\newcommand{{\Hl}}{{H^{\ell}}} 
\newcommand{{\mHop}}{{m_{H^{\rm op}}}} 
\newcommand{{\Hop}}{{H^{\rm op}}} 
\newcommand{{\mUop}}{{m_{U^{\rm op}}}} 
\newcommand{{\mUopp}}{{m_{\scriptscriptstyle{U^{\rm op}}}}} 
\newcommand{{\Uop}}{{U^{\rm op}}}
\newcommand{{\mVop}}{{m_{V^{\rm op}}}} 
\newcommand{{\Vop}}{{V^{\rm op}}}  
\newcommand{{\Ae}}{{A^{\rm e}}}
\newcommand{{\Be}}{{B^{\rm e}}}
\newcommand{{\Ue}}{{U^{\rm e}}}
\newcommand{{\He}}{{H^{\rm e}}}
\newcommand{{\Aop}}{{A^{\rm op}}}
\newcommand{{\Aope}}{({A^{\rm op}})^{\rm e}}
\newcommand{{\Aopl}}{{A^{\rm op}_\pl}}

\newcommand{{\Bop}}{{B^{\rm op}}}
\newcommand{{\Bopp}}{{\scriptscriptstyle{{B^{\rm op}}}}}
\newcommand{{\Bope}}{({B^{\rm op}})^{\rm e}}
\newcommand{{\Bpl}}{{B_\pl}}

\newcommand{{\op}}{{{\rm op}}}
\newcommand{{\coop}}{{{\rm coop}}}
\newcommand{{\sop}}{{*^{\rm op}}}
\newcommand{{\co}}{{{\rm co}}}

\newcommand{\kmod}{k\mbox{-}\mathbf{Mod}}                     %
\newcommand{\hmod}{H\mbox{-}\mathbf{Mod}}                     

\newcommand{{\gog}}{{G \rightrightarrows G_0}}

\newcommand{{\rra}}{\rightrightarrows}

\newcommand{{\lra}}{\ \longrightarrow \ }
\newcommand{{\lla}}{\ \longleftarrow \ }
\newcommand{{\lma}}{\ \longmapsto \ }

\newcommand{{\bull}}{{\scriptscriptstyle{\bullet}}}
\newcommand{{\qqquad}}{{\quad\quad\quad}}

\newsavebox{\foobox}

\newcommand{\mancino}{{\,\scalebox{0.7}{\rotatebox{90}{\mancone}}\,}}

\begin{document}

\title{Cyclic duality between BV algebras and BV modules} 

\author{Niels Kowalzig}

\begin{abstract}
%
  We show that
  if an operad is at the same time a cosimplicial object such that the respective structure maps are compatible with the operadic composition in a natural way, then one obtains a Gerstenhaber algebra structure on cohomology, and if the operad is cyclic, even that of a BV algebra. In particular,
  if a cyclic opposite module over an operad with multiplication is itself a cyclic operad that meets the cosimplicial compatibility conditions, the cohomology of its cyclic dual turns into a BV algebra. This amounts to conditions for when the cyclic dual of a BV module is endowed with a BV algebra structure, a result we exemplify by looking at classical and less classical (co)homology groups in Hopf algebra theory. 
\end{abstract}

\address{Dipartimento di Matematica, Universit\`a di Roma Tor Vergata, Via della Ricerca Scientifica 1, 00133 Roma, Italy}

\email{niels.kowalzig@uniroma2.it}

%

\keywords{
  Operads, cyclic opposite modules, cyclic duals, Gerstenhaber algebras, BV algebras, derived functors, Hopf algebras.
}

\subjclass[2020]{
18M65, 18M85, 16E40, 18G15.
}

\maketitle

\setcounter{tocdepth}{2}
\tableofcontents

\section*{Introduction}

\addtocontents{toc}{\protect\setcounter{tocdepth}{1}}

Higher structures, such as brackets and products on cohomology groups or on cochain spaces only, extensively appear in many contexts in the fields of algebra, geometry, topology and
mathematical physics. For example, {\em Batalin-Vilkoviski\u\i} algebras, as special cases of Gerstenhaber algebras equipped with a degree $+1$ differential whose bracket measures the failure
of this differential to be a (graded) derivation of the cup product, were 
originally introduced in quantum field theory to deal with path integrals in the presence of symmetries. In algebraic
topology, apart from their massive appearance in the aforementioned study of higher structures on an abstract level, 
BV algebras (or their homotopy versions) naturally emerge in applications to string topology, Poisson geometry, algebraic deformation theory, and many more.

\subsection*{Aims and objectives}

This article investigates, on a formal level, what was observed on examples in \cite{Kow:ANCCOTCDOE}: that for a pair composed by a BV (short for {\em Batalin-Vilkoviski\u\i}) algebra and a BV module (see below) over it, their respective cyclic duals turn these r\^oles around, that is to say, the BV module becomes a BV algebra and the BV algebra one started with now yields a BV module over the latter. Pairs formed by Gerstenhaber algebras (with BV algebras seen as their stronger versions) and BV modules over these are known under the term {\em noncommutative differential} or {\em Cartan calculus}, in the sense of Rinehart \cite{Rin:DFOGCA} or Nest-Ta\-mar\-kin-Tsygan \cite{NesTsy:OTCROAA, TamTsy:NCDCHBVAAFC, Tsy:CH}.

More precisely, as shown in \cite{Kow:GABVSOMOO},
the datum of a cyclic
unital opposite module $\cM$ over an operad $\cO$ with multiplication, that is, a sequence 
$\{\cM(n)\}_{n \geq 0}$ of $k$-modules equipped with a collection of maps
$$
\bullet_i\colon \cO(p) \otimes \cM(n) \to \cM(n-p+1),  \quad 0 \leq i \leq n-p+1,
$$
that are associative in a sense (see Definition \ref{molck} for full details) and compatible with a degree preserving operator $t$, the cyclic operator, is sufficient to ob\-tain on the couple $(\cO, \cM)$ the structure of a {\em homotopy} noncommutative differential calculus: as an illustration, one might think of the chain complex
computing $\Tor$ groups over the cochain complex computing $\Ext$ groups in the realm of Hopf algebras or Hopf algebroids, and hence for associative algebras in Hoch\-schild theory as well.
A (homotopy) calculus implies, in particular, the existence of certain (homotopy) structure maps such as a {\em cap product} as well as of a {\em Lie derivative} on $\cM$ along $\cO$, which, up to homotopy terms, obey relations analogous to those by Cartan known in classical differential geometry.

In such a situation, one might want to call (for reasons that are obvious from the construction) the opposite module $\cM$ a {\em homotopy BV module} over the operad $\cO$. Descending to (co)ho\-mo\-lo\-gy, the groups $H_\bull(\cM)$ are, in this spirit, called a {\em BV module} over $H^\bull(\cO)$ since this construction depends on the cyclic operator $t$ or rather on the cyclic (Connes-Rinehart-Tsygan) boundary $B$ of degree $-1$ induced by $t$, which is the dual construction to the $+1$ differential mentioned above appearing in BV algebras (sometimes, see \cite{Get:BVAATDTFT}, denoted by $\gD$ but $B$ would be better-grounded, at least  from this perspective).

Based on this construction, the pattern  observed in \cite{Kow:ANCCOTCDOE} with respect to the {\em cyclic dual}, a duality notion originating from the self-duality of the cyclic category \cite{Con:CCEFE}, by examining examples coming from Hopf algebroid theory was quite remarkable: as a homotopy BV module, the graded $k$-module $\cM$ is, in particular, a cyclic $k$-module the underlying simplicial structure of which is induced by the multiplication datum in the operad $\cO$. As such, it has a cyclic dual $\hat\cM$, which, by construction, comes with the structure of a cocyclic $k$-module, inevitably induced again by the multiplication structure of the operad $\cO$; see \S\ref{briefbrief} for a brief account on cyclic duality and Lemma \ref{Bop} for the specific situation under consideration here. In the examples under examination (think of $\Cotor$ groups as cyclic duals to $\Tor$ groups), the graded $k$-module underlying the cochain complex given by $\hat\cM$ forms itself an op\-erad, and the respective co\-ho\-mo\-lo\-gy groups $H^\bull(\hat\cM)$ turn out not only to be a Gerstenhaber but also a BV algebra; to sum up, the cyclic dual of a (homotopy) BV module yields a (homotopy) BV algebra. Formalising this to a general level is, as said, the main motivation for the article at hand.

Adding, moreover, the assumption that the operad $\cO$ is not only multiplicative but cyclic as well, and hence yields a cocyclic $k$-module, too,
one can pass to the cyclic duals {\em both} for $\cO$ and $\cM$ and, as furthermore observed in {\em op.~cit.}, these exchange their r\^oles:
whereas before $H_\bull(\cM)$ was a BV module over $H^\bull(\cO)$ and hence the pair $\big(H^\bull(\cO), H_\bull(\cM) \big)$ formed a noncommutative calculus, now, after passing to cyclic duals, $H_\bull(\hat\cO)$ defines a BV module over $H^\bull(\hat\cM)$, and hence the couple $\big(H^\bull(\hat\cM), H_\bull(\hat\cO) \big)$ yields a noncommutative differential calculus as well. As a concrete example, one should think of the pair $(\Ext, \Tor)$ being transformed into the pair $(\Cotor, \Coext)$. Formalising this observation to a general level as well in order to obtain a complete duality between BV algebras and BV modules turns out to be more intricate: this should be based, roughly speaking, on a sort of {\em coloured} (two-sided) opposite action, as a collection of maps
$$
\cO(p) \otimes \cM(n) \to
\begin{cases}
  \cM(n-p+1) & \mbox{if} \ p <n,
  \\
  \cO(p-n+1) & \mbox{if} \ n <p,  
  \end{cases}
$$
with special attention to the case $p=n$, such that not only both colours define the respective structure of a (cyclic) left resp.\ right opposite module over an operad but at the same time are also compatible with both underlying operadic compositions on $\cO$ and $\cM$ in a natural way. This approach is presently
under closer examination.


\subsection*{Main results}
It is well-known that a (nonsymmetric) operad with multiplication (in the category of $k$-modules) induces a cosimplicial structure whose cohomology groups form a Gerstenhaber algebra \cite{Ger:TCSOAAR, GerVor:HGAAMSO}; if the operad is cyclic with compatible multiplication, then these, in particular, produce a BV algebra structure \cite{Men:BVAACCOHA}. Here, the crucial observation is that these two results essentially only depend on the fact that the cofaces and codegeneracies (induced by the operad multiplication) are {\em compatible} in a natural way with the vertical operadic composition, a property which we generalise in Definition \ref{compitompi} to {\em any}
cosimplicial structure, {\em i.e.}, with respect to possibly preexistent cofaces and codegeneracies not necessarily originating from an operad multiplication. This allows us to prove (we refer to the main text for notation and definitions) in Corollaries \ref{alsomain} \& \ref{alsoalsomain}:

\begin{theorem*}
  The cohomology groups of a cosimplicial-compatible nonsymmetric operad (in $\kmod$)
  form a Gerstenhaber algebra. If the operad is cyclic as well, then this Gerstenhaber algebra is Batalin-Vilkoviski\u\i. 
  \end{theorem*}

In more detail, this originates from a stronger result already on the level of cochains, {\em i.e.}, from a homotopy formula: let $\cP$ be a nonsymmetric operad equipped with the structure of a cocyclic $k$-module $(\cP^\bull, \gd_\bull, \gs_\bull, \tau)$. Then defining $\gd$ as usual as the (alternating) sum over all cofaces, along with a cup product resp.\ a cyclic boundary of degree $-1$,
$$
 x \cup y := (\gd_0 y) \circ_{1} x, \qquad   B x  :=  \Sum^{p-1}_{i=0} (-1)^{\sss \sss (p-1)(i-1)} \tau^{-i-1} ( \gs_{p-1} \tau x),
$$
 along with a higher $B$-operation in the form of a degree $-2$ homotopy operator,
 $$
  S_x y :=
\Sum^{p-1}_{i=1} \Sum^i_{j=1} (-1)^{\scriptscriptstyle (q-1)i + (p+q)j + pq}  \, \tau^{-j} (\gs_{p-1} \tau x) \circ_{p-i+j-1} y
 $$
 for any $x,y$ of degree $p$ resp.\ $q$ on the normalised complex induced by $\cP$, the Gerstenhaber bracket  can be expressed as
 \begin{equation*}
\begin{split}
  \{x,y\} &\, = \,
(-1)^{\sss (q-1)p} B(y \cup x)
    - (-1)^{\sss p} Bx \cup y
        - (-1)^{\sss p(q-1)} By \cup x  
\\
& \, \, \quad
+ (-1)^{\sss pq} \delta(S_xy)  + (-1)^{\sss p} S_{\gd x} y - (-1)^{\sss pq} S_{x}\gd y
\\
  & \, \, \quad  - (-1)^{\sss q} \delta(S_y x)  + (-1)^{\sss p(q-1)} S_{\gd y} x  - (-1)^{\sss p+q} S_y \gd x,
\end{split}
\end{equation*}
 that is, the bracket is generated by $B$ and the cup product up to homotopy.

 A situation where this occurs is by considering the cyclic dual $\hat\cM$ of a cyclic opposite module $(\cM, \bullet_i, t)$ over an operad with multiplication $(\cO, \mu, e)$ as mentioned at the beginning, which becomes a cosimplicial $k$-module by means of
 $
 \gd_i x = e \bullet_i x$ and $\gs_j x = \mu \bullet_j x$,
see Lemma \ref{Bop} for details and notation. If $\hat\cM$ happens to be itself an operad that is cosimplicial and cocyclic-compatible, then we can define, analogously to the above,
$$
x \cup y := (e \bullet_0 y) \circ_{1} x, \quad  B x :=  \Sum^{p-1}_{i=0} (-1)^{\sss \sss (p-1)(i-1)} t^i ( \mu \bullet_0 tx),
$$
along with
$$
  S_x y :=
\Sum^{p-1}_{i=1} \Sum^i_{j=1} (-1)^{\scriptscriptstyle (q-1)i + (p+q)j + pq}  \, t^{j-1} (\mu \bullet_0 t x) \circ_{p-i+j-1} y
$$
to obtain the analogue of the above homotopy formula for the respective Gerstenhaber bracket, which proves the existence of a BV structure on the cohomology groups $H^\bull(\hat\cM)$ of the cyclic dual of the cyclic opposite module. We can therefore state in Proposition \ref{gerstasgerstcan} \& Theorem \ref{main}:

\begin{theorem*}
  If the cyclic dual of a cyclic opposite module over an operad with multiplication is itself a cosimplicial and cocyclic-compatible operad, then its cohomology groups carry the structure of a Batalin-Vilkoviski\u\i\ algebra.
  \end{theorem*}

In the final example section, see \S\ref{examples}, we show how to deploy these general constructions in order to
show
that both the cohomology groups $\Ext_H(k,k)$ as well as $\Cotor_H(k,k)$ for a Hopf algebra over a commutative ring $k$ with involutive antipode are not only Gerstenhaber algebras but, in particular, BV.

\subsection*{Notation and conventions}
\label{schonweniger}
In all what follows, $k$ denotes a commutative ring, sometimes a field, usually of characteristic zero; denote by $\kmod$ the category of $k$-modules. As customary,
unadorned tensor products (or Homs) are meant to be over $k$.
The term {\em operad} usually refers to nonsymmetric operads in $\kmod$, see below.

\subsection*{Acknowledgements}
Partially supported by the MIUR Excellence Department Project MatMod@TOV
(CUP:E83C23000330006) and by the PRIN 2017 {\it Real and Complex Manifolds: Topology, Geometry, and Holomorphic Dynamics}, Ref.~2017JZ2SW5. The author is a member of the {\em Gruppo Nazionale per le Strutture Algebriche, Geometriche e le loro
Applicazioni} (GNSAGA-INdAM).

\section{Preliminaries}
\label{prel}

\subsection{Operads, Gerstenhaber and BV algebras}
\label{pamukkale1}

The following is a brief summary of standard material needed in the sequel, and which can be found, for example, in \cite{Ger:TCSOAAR, GerVor:HGAAMSO, Kad:AIASICATRHT, Get:BVAATDTFT}, or elsewhere.
A {\em (non\-symmetric) operad}~$\cO$ (in the category 
of $k$-modules) is a sequence $\{\cO(n)\}_{n \geq 0}$ of $k$-modules 
with $k$-bilinear operations $\circ_i \colon \cO(p) \otimes \cO(q) \to \cO({p+q-1})$ for $i = 1, \ldots, p$,
subject to: 
\begin{eqnarray}
\label{danton}
\nonumber
\gvf \circ_i \psi &=& 0 \qquad \qquad \qquad \qquad \qquad \! \mbox{if} \ p < i \quad \mbox{or} \quad p = 0, \\
(\varphi \circ_i \psi) \circ_j \chi &=& 
\begin{cases}
(\varphi \circ_j \chi) \circ_{i+r-1} \psi \qquad \mbox{if} \  \, j < i, \\
\varphi \circ_i (\psi \circ_{j-i +1} \chi) \qquad \hspace*{1pt} \mbox{if} \ \, i \leq j < q + i, \\
(\varphi \circ_{j-q+1} \chi) \circ_{i} \psi \qquad \mbox{if} \ \, j \geq q + i.
\end{cases}
\end{eqnarray}
Call an operad {\em unital} if there is an element $\mathbb{1} \in \cO(1)$ such that 
$
\gvf \circ_i \mathbb{1} = \mathbb{1} \circ_1 \gvf = \gvf
$ 
for all $\gvf \in \cO(p)$ and $i \leq p$, and call it {\em with multiplication} if there is an element  $\mu \in \cO(2)$ plus an element $e \in \cO(0)$ such that $\mu \circ_1 \mu = \mu \circ_2 \mu$ and 
$\mu \circ_1 e = \mu \circ_2 e = \mathbb{1}$.

An operad with multiplication  $(\cO, \mu, e)$ turns into a cosimplicial $k$-module with cofaces given by $\gd_0 \gvf := \mu \circ_2 \gvf$, $\gd_i \gvf := \gvf \circ_{i} \mu$ for $i = 1, \ldots, p$, and $\gd_{p+1} \gvf := \mu \circ_1 \gvf$ for $\gvf \in \cO(p)$, together with the codegeneracies $\sigma_j \gvf := \gvf \circ_{j+1} e$ for $j = 0, \ldots, p-1$. This induces a cochain complex denoted again by $\cO$ or $\cO^\bull$ given by $\cO^p := \cO(p)$ in degree $p$, 
 and with differential $\gd \colon \cO(p) \to \cO({p+1})$ defined, as usual, by the alternating sum $\gd := \sum^{p+1}_{i=0} (-1)^i \gd_i$ over all cofaces, and with co\-ho\-mology
$
H^\bullet(\cO) := H(\cO^\bullet, \gd).
$ 
Moreover, 
one can form the {\em cup product}
\begin{equation}
  \label{cupye}
  \gvf \smallsmile \psi := (\mu \circ_2 \psi) \circ_1 \gvf \ \in \cO(p+q)
\end{equation}
for $\gvf \in \cO(p)$ and $\psi \in \cO(q)$,
which turns $(\cO, \smallsmile, \gd)$ into a dg algebra. Defining 
the {\em braces}
\begin{equation}
  \label{bracytracy}
\varphi\{\psi\} := \Sum^{p}_{i=1}
(-1)^{\sss (q-1)(i-1)} \varphi \circ_i \psi  \ \in \cO({p+q-1})
  \end{equation}
finally allows for the construction of the {\em (Gerstenhaber) bracket}
\begin{equation}
  \label{bracket-po-packet}
{\{} \varphi,\psi \}
:= \varphi\{\psi\} - (-1)^{\sss (p-1)(q-1)} \psi\{\varphi\}.
\end{equation}
Descending to cohomology leads to the well-known fact, see \cite{Ger:TCSOAAR}, that the triple $(H^\bullet(\cO), \smallsmile, \{\cdot, \cdot\})$ forms a {\em Gerstenhaber algebra} over $k$.
By this we mean, in general, a $k$-algebra $V$ with a graded commutative product $\cup$ of degree zero, together with a graded Lie bracket $\{\cdot, \cdot\}$ of degree $-1$ such that the {\em Leibniz identity} 
\begin{equation}
  \label{leibniz}
\{ x, y \cup z\} = \{x, y\} \cup z + (-1)^{\sss (x-1)y} y \cup \{x, z\} 
\end{equation}
holds for all $x,y,z \in V$. Finally, a {\em BV algebra}, short for {\em Batalin-Vilkoviski\u\i\ algebra}, is a Gerstenhaber algebra whose bracket is generated by a differential $B$ of degree $-1$ and the (cup) product, that is
    \begin{equation}
        \label{bvbv}
    \{x,y\}
 = (-1)^{\sss x} B(x \cup y) - (-1)^{\sss x} Bx \cup y  - x \cup B y,
    \end{equation}
    see, for example, \cite[Prop.~1.2]{Get:BVAATDTFT}.
Alternatively, a BV algebra is a Gerstenhaber algebra whose bracket measures, up to a sign, the failure of $B$ being a derivation of the cup product. A (by now) fundamental result \cite{Men:BVAACCOHA} states that if an operad $\cO$ with multiplication is {\em cyclic} in the sense of Eq.~\eqref{unschoen}, to be discussed later, and if the cyclic operator is compatible with the multiplication as specified in {\it op.~cit.}, then its cohomology $H^\bullet(\cO)$ is not only a Gerstenhaber but moreover a BV algebra.

\subsection{Opposite modules and cyclic opposite modules}
In this subsection we are going to recall {\em opposite} modules over operads as introduced in \cite{Kow:GABVSOMOO}, which differ from the customary operad modules by the fact that when acting with the operad on it, trees become {\em smaller}, that is, the number of leaves is going to be reduced, not increased. We remark here that passing to negative degrees does not produce one notion out of the other; it is rather the $k$-linear dual that mediates between the two notions.
More precisely:

\begin{definition}
  \label{molck}
  Let $\cO$ be a (unital) operad.
  \begin{enumerate}
    \compactlist{99}
    \item
An {\em opposite (unital, left) $\cO$-module} consists of
a family of $k$-modules $\cM = \{ \cM(n) \}_{n \geq 0}$ endowed with $k$-linear
operations
$$
\qquad        \bullet_i \colon 
        \cO(p) \otimes \cM(n) \to \cM({n-p+1}), \qquad 
1 \leq i \leq  n- p +1, \ \ 0 \leq p \leq n, 
$$
defined to be zero if $p > n$, such that for $\gvf \in \cO(p)$, $\psi \in \cO(q)$, and $x \in
\cM(n)$
\begin{eqnarray}
\label{SchlesischeStr}
\gvf \bullet_i (\psi \bullet_j x) \!\!\!\!&=\!\!\!\!& 
\begin{cases} 
\psi \bullet_j (\gvf \bullet_{i + q - 1}  x)  & \mbox{if} \ j < i, 
\\
(\gvf \circ_{j-i+1} \psi) \bullet_{i}  x  & \mbox{if} \ j - p < i \leq j, \\
\psi \bullet_{j-p + 1} (\gvf \bullet_{i}  x)  & \mbox{if} \ 1 \leq i \leq j - p,
\end{cases} 
\\[1mm]
\label{SchlesischeStr-1}
\mathbb{1} \bullet_i x \!\!\!\!&=\!\!\!\!&  x \hspace*{3cm} \, \mbox{for \ } i = 0, \ldots, n,
\end{eqnarray}
holds for $p, q, n \geq 0$, and where for $p=0$ the relations \eqref{SchlesischeStr} are understood without the middle line.
\item
An opposite $\cO$-module is called {\em cyclic} if 
there is an
an {\em extra} $k$-linear composition map
$$
\bullet_0 \colon \cO(p) \otimes \cM(n) \to \cM({n-p+1}), \quad 0 \leq p \leq n+1,
$$
defined to be zero if $p > n+1$, such that Eqs.~\eqref{SchlesischeStr}--\eqref{SchlesischeStr-1} hold as well for the case $i= 0$ or $j=0$, plus a 
degree-preserving morphism $t\colon \cM(n) \to \cM(n)$ for all $ n \geq 1$ such that
\begin{equation}
\label{lagrandebellezza1}
t(\gvf \bullet_{i} x) = \gvf \bullet_{i+1} t(x),  \qquad i = 0, \ldots, n-p,
\end{equation}
is true for $\gvf \in \cO(p)$ and $x \in \cM(n)$, and such that, finally,
\begin{equation}
\label{lagrandebellezza2}
t^{n+1} = \id
\end{equation}
holds on $\cM(n)$.
\end{enumerate}
  \end{definition}
%
%
\noindent The following figure illustrates the idea of an opposite $\cO$-module:
\begin{center}
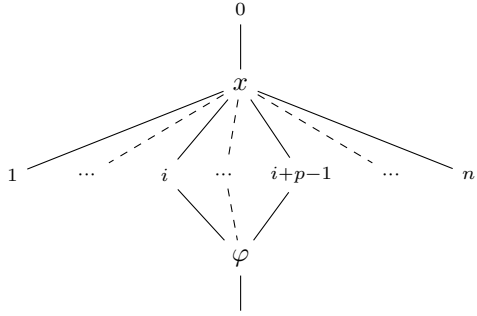

  \label{bretzel}
\begin{tikzpicture}
\node(0) at (0, 0) {$\foo 0$};
\node(x) at (0,-1) {$x$};
\node(1) at (-3,-2.2) {$\foo 1$};
\node(2) at (-2,-2.2) {$\foo \cdots$};
\node(i) at (-1,-2.2) {$\foo i$};
\node(4) at (-0.2,-2.2) {$\foo \cdots$};
\node(i+p-1) at (0.8,-2.2) {$\foo i+p-1$};
\node(6) at (2,-2.2) {$\foo \cdots$};
\node(n) at (3,-2.2) {$\foo n$};
\node(phi) at (0,-3.3) {$\varphi$};

\draw(x)--(1);
\draw[dashed](x)--(2);
\draw(x)--(i);
\draw[dashed](x)--(4);
\draw(x)--(i+p-1);
\draw[dashed](x)--(6);
\draw(x)--(n);
\draw(i)--(phi);
\draw[dashed](4)--(phi);
\draw(i+p-1)--(phi);
\draw(0)--(x);

\draw(phi)--(0,-4);
\end{tikzpicture}
\captionof{figure}{The operation $\gvf \bullet_i x$ on opposite modules.}\label{tratto}
\end{center}
The next figure gives a graphical understanding of the condition \eqref{lagrandebellezza1} involving the cyclic operator $t$ in case of a cyclic opposite $\cO$-module:

\begin{center}
\begin{tikzpicture}
\node(0) at (-6, -3) {$\foo 0$};
\node(x) at (-4,-1.4) {$x$};
\node(2) at (-5.5,-3) {$\foo \cdots$};
\node(i) at (-5,-3) {$\foo i$};
\node(4) at (-4.5,-3) {$\foo \cdots$};
\node(i+p-1) at (-3.7,-3) {$\foo i+p-1$};
\node(6) at (-2.7,-3) {$\foo \cdots$};
\node(n-1) at (-2,-3) {$\foo n-1$};
\node(n) at (-4,0) {$\foo n$};
\node(phi) at (-4.5,-4.3) {$\varphi$};

\draw[dashed](x)--(2);
\draw(x)--(i);
\draw[dashed](x)--(4);
\draw(x)--(i+p-1);
\draw[dashed](x)--(6);
\draw(x)--(n-1);
\draw(i)--(phi);
\draw[dashed](4)--(phi);
\draw(i+p-1)--(phi);

\draw(phi)--(-4.5,-5.3);

\draw  (-3.7,-1.5) to [out=330,in=270] (-1.5,-2) to [out=90,in=270] (-4,-0.2) ;
\draw (x) to [out=90, in=38] (-5.8, -2.8);

\node(0) at (1, -3) {$\foo 0$};
\node(x) at (3,-1.4) {$x$};
\node(2) at (1.5,-3) {$\foo \cdots$};
\node(i+1) at (2.2,-3) {$\foo i+1$};
\node(4) at (2.85,-3) {$\foo \cdots$};
\node(i+p) at (3.5,-3) {$\foo i+p$};
\node(6) at (4.3,-3) {$\foo \cdots$};
\node(n-1) at (5,-3) {$\foo n-1$};
\node(n) at (3,0) {$\foo n$};
\node(phi) at (2.8,-4.3) {$\varphi$};
\node(=) at (-0.3, -3) {$=$};
\node(t) at (-1, -4.5) {$t$};
\node(t) at (6, -3.4) {$t$};

\draw[dashed](x)--(2);
\draw(x)--(i+1);
\draw[dashed](x)--(4);
\draw(x)--(i+p);
\draw[dashed](x)--(6);
\draw(x)--(n-1);
\draw(i+1)--(phi);
\draw[dashed](4)--(phi);
\draw(i+p)--(phi);

 \draw(phi)--(2.8,-5.3);

\draw  (3.3,-1.5) to [out=330,in=270] (5.5,-2) to [out=90,in=270] (3,-0.2) ;
\draw (x) to [out=90, in=38] (1.2, -2.8);

\draw [thick, rounded corners] (-6.3, -0.5) rectangle (-1.2, -4.8);
\draw [thick, rounded corners] (0.7, -0.5) rectangle (5.8, -3.7);

 \end{tikzpicture}
\captionof{figure}{The relation $t(\gvf \bullet_{i} x) = \gvf \bullet_{i+1} t(x)$. 
}
\label{rel}
 \end{center}
In particular \cite[Prop.~3.5]{Kow:GABVSOMOO}, if the operad in question is an operad with multiplication
$(\cO, \mu, e)$, then a cyclic
opposite $\cO$-module 
$\cM$ turns into a cyclic $k$-module
via the cyclic operator $t$ and simplicial structure described by the faces $d_i\colon \cM(n) \to \cM({n-1})$ and degeneracies $s_j \colon \cM(n) \to \cM({n+1})$ defined by 
\begin{equation}
\label{colleoppio}
\begin{array}{rcll}
d_i(x) & = & \mu \bullet_{i} x, & i = 0, \ldots, n-1, \\
d_n(x) & = & \mu \bullet_0 t (x), & \\
s_j (x)& = & e \bullet_{j+1} x, & j = 0, \ldots, n, \\
\end{array}
\end{equation}
where $x \in \cM(n)$.
Observe that, as is the case for any cyclic $k$-module, one can produce an extra degeneracy $s_{-1} := t \, s_n$ which, in this construction, is given precisely by $e \bullet_0 -$ which has not been used so far in \eqref{colleoppio}.
Indeed, 
\begin{equation}
  \begin{split}
\label{cine40}
s_{-1}(x) = t \, s_n(x) &= t(e \bullet_{n+1} x)
\\
&=  t(e \bullet_{n+1} t^{n+1}x)
=  t^{n+2}(e \bullet_0 x) = e \bullet_0 x,
\end{split}
  \end{equation}
where we used \eqref{lagrandebellezza2} in step three and \eqref{lagrandebellezza1} in the fourth.



\subsection{Cyclic duality}
\label{briefbrief}
An interesting feature of 
Connes' cyclic category 
$\gL$ is its self-duality, that is to say, $\gL = \gL^\op$, and therefore
cyclic and cocyclic $k$-modules can be turned one into the other, {\em cf.}\ \cite{Con:CCEFE}, which is not possible merely on the simplicial level.
However, there are infinitely many
ways to do this due to the autoequivalences of the cyclic category \cite[6.1.14 \& E.6.1.5]{Lod:CH}.
We shall mainly use the
following standard convention to go from cyclic to cocyclic $k$-modules.

\begin{definition}
  \label{moulesfrites}
The {\em cyclic dual} of  
a cyclic $k$-module
$X_\bull = (X_\bull, d_\bull,
s_\bull, t_\bull)$ is
the cocyclic $k$-module 
$\hat X^\bull := (\hat X^\bull, \gd_\bull,\gs_\bull,
\tau_\bull)$, where $\hat X^n := X_n$, and
%
%
\begin{equation}
\label{moulesfrites1}
\begin{array}{rrrccll}
\gd_i \!\!\!\! & := \!\!\!\! & s_{i-1} \colon  & \hat X^n & \to & \hat X^{n+1}, & \quad 0 \leq i \leq n+1, \\
\gs_j \!\!\!\! & := \!\!\!\! & d_j \colon  & \hat X^n & \to & \hat X^{n-1}, & \quad 0 \leq j \leq n-1, \\
\tau_n \!\!\!\! & := \!\!\!\!  & t^{-1}_n  \colon  & \hat X^n & \to & \hat X^{n}, &
\end{array}
\end{equation}
in degree $n$.
\end{definition}

\begin{rem}
In case $i = 0$, this prescription means that the zeroth coface is given by the extra degeneracy $s_{-1} = t \, s_n$ as defined above.
In examples, one might encounter the situation in which the cyclic operator $t$ is not necessarily invertible. In such a situation, among other possibilities, the following convention can be used:
\begin{equation}
\label{moulesfrites2}
\begin{array}{rrrccll}
  \gd_i \!\!\!\! & := \!\!\!\! & s_{n-i} \, \colon  & \hat X^n & \to & \hat X^{n+1}, & \quad 0 \leq i \leq n+1,
  \\
\gs_j \!\!\!\! & := \!\!\!\! & d_{n-j} \, \colon  & \hat X^n & \to & \hat X^{n-1}, & \quad 0 \leq j \leq n-1, \\
\tau_n \!\!\!\! & := \!\!\!\!  & t_n  \, \colon  & \hat X^n & \to & \hat X^{n}. &
\end{array}
\end{equation}
Note that
in \eqref{moulesfrites1}
  the last face map $d_{n}$ is not
  used in the construction of the cyclic dual, whereas in \eqref{moulesfrites2} the first zeroth face map $d_0$ is not required:
there is one less codegeneracy on 
 $\gs_i \colon \hat X^{n} \rightarrow \hat X^{n-1}$ than there are faces
 $d_i \colon X^{n}
\rightarrow X^{n-1}$.
Conversely, in both cases there are
not enough degeneracies to derive all
coface maps, which is where the extra degeneracy
$
s_{-1} = t s_n \colon X_n \to X_{n+1}
$
comes into play.
\end{rem}

\begin{rem}
\label{normalised}
Most of the time, both for chain or cochain complexes, we are going to work on the
respective {\em normalised} complexes. In the chain case,  by this we mean the quotient of the original complex by the (acyclic) subcomplex spanned by the
images of the degeneracy maps of the simplicial 
$k$-module given by Eqs.~\eqref{colleoppio}, that is, given by the cokernel of the degeneracy maps $s_j$. 
Similarly, for the cochain complex defined by the cosimplicial $k$-module obtained from Eqs.~\eqref{moulesfrites1}, one considers the intersection of the kernels of the codegeneracies $\gs_j$. In order not to overload our notation we will not distinguish in notation a (co)chain complex from its normalised one.  
  \end{rem}

\begin{lemma}
  \label{Bop}
  The cyclic dual $\hat\cM$ given by \eqref{moulesfrites1} of a cyclic opposite module $(\cM, t)$ over an operad $(\cO, \mu, e)$ with multiplication becomes a cosimplicial $k$-module via the following cofaces and codegeneracies:
  \begin{equation}
\label{journuit}
\gd_i x = e \bullet_i x, \qquad \gs_j x = \mu \bullet_j x, \qquad 0 \leq i \leq n+1, \  0 \leq j \leq n-1,
  \end{equation}
where $x \in \cM(n)$.  
Then, $\cM^\bull := \cM(\bullet)$ becomes a 
  cochain complex the standard way by means of the coboundary $\gd \colon \hat\cM^n \to \hat\cM^{n+1}$, 
    \begin{equation}
    \label{stakker}
  \gd x = \Sum\limits_{i=0}^{n+1} (-1)^i \, e \bullet_i x,
    \end{equation}
    of degree $+1$,
    that is, the alternating sum over all cofaces.
Moreover, by means of
the operator $B \colon \hat\cM^n \to \hat\cM^{n-1}$ of degree $-1$ given by 
\begin{eqnarray}
  \label{humanoid1}
  B x & = & \Sum^{n-1}_{i=0} (-1)^{\sss (n-1)(i-1)} t^i ( \mu \bullet_0 tx)
\end{eqnarray}
on the normalised complex,
one obtains the structure on a mixed complex $(\hat\cM, \gd, B)$, that is, simultaneously a cochain and a chain complex such that
$$
\gd^2 = 0, \quad B^2 = 0, \quad \gd B + B \gd = 0
$$
holds.
  \end{lemma}
  
\begin{proof}
  All statements are essentially automatic by the very construction of the cyclic dual together with the fact that the identities \eqref{colleoppio} define the structure of a simplicial $k$-module, and together with the operator $t$ even the structure of a cyclic $k$-module. We only have to show that the explicit expression for $B$ corresponds to the customary Connes-Rinehart-Tsygan operator for cocyclic $k$-modules.
  Indeed, with respect to the operators defined in \eqref{moulesfrites1} for the cyclic dual, one has, on the normalised complex 
  $$
  Bx =  \Sum^{n-1}_{i=0} (-1)^{\sss (n-1)i } \tau^i (\gs_{-1} x),
  $$
  following the standard way the cyclic boundary is constructed as for example detailed in \cite[2.5.10]{Lod:CH} in its homological version.
 With $\gs_{-1} := \gs_{n-1} \tau_n$, this amounts to
  \begin{eqnarray*}
    Bx &=&  \Sum^{n-1}_{i=0} (-1)^{\sss (n-1)i } \tau^i (\gs_{-1} x)
    \\
    &=&  \Sum^{n-1}_{i=0} (-1)^{\sss (n-1)i } t^{-i} \big(d_{n-1} (t^{-1} x)\big)
     \\
     &=&  \Sum^{n-1}_{i=0} (-1)^{\sss (n-1)i } t^{-i} (\mu \bullet_{n-1} t^{-1} x)
        \\
        &=&  \Sum^{n-1}_{i=0} (-1)^{\sss (n-1)i } t^{n-i}(\mu \bullet_{n-1} t^{n} x)
          \ \ = \ \ \Sum^{n-1}_{i=0} (-1)^{\sss (n-1)(i-1) } t^i (\mu \bullet_0 t x)
             \end{eqnarray*}
 by means of $t^{n+1}x = x$ in degree $n$ and by re-indexing $i \mapsto n-i-1$ in the last line; hence, Eq.~\eqref{humanoid1} as claimed.
  \end{proof}
We now have all the necessary ingredients at hand to analyse the relationship between cyclic duals and BV algebras.

\section{Cyclic duality as tree reflexion}

\subsection{Cosimplicial and cocyclic compatibility}
Let $(\cO, \mu, e)$ be an operad with multiplication in the sense specified in \S\ref{pamukkale1}. Studying the proof in \cite{Ger:TCSOAAR} why the cohomology $H(\cO)$ becomes a Gerstenhaber algebra with respect to the bracket \eqref{bracket-po-packet} and the cup product \eqref{cupye}, one observes that this essentially hinges upon the sort of associativity of the operadic vertical composition as~in Eqs.~\eqref{danton}, and, in particular, upon the respective relations for the special elements $\mu \in \cO(2)$ and $e \in \cO(0)$.

For example, from Eqs.~\eqref{danton} one deduces that for the multiplication element $\mu \in \cO(2)$ and
$\phi \in \cO(p)$, $\psi \in \cO(q)$, as well as $1 \leq j \leq p$, 
\begin{equation}
  \label{useless1}
  \hspace*{1cm}
  \begin{array}{rclll}
    \mu \circ_2 (\phi \circ_j \psi) &=&  (\mu \circ_2 \phi) \circ_{j+1} \psi, &
    \\
    (\phi \circ_j \psi) \circ_i \mu &=&  (\phi \circ_i \mu) \circ_{j+1} \psi &  \mbox{if} \ \ 1 \leq i \leq j-1,
    \\
    (\phi \circ_j \psi) \circ_i \mu &=&  \phi \circ_j (\psi \circ_{i-j+1} \mu) &  \mbox{if} \ \ j \leq i \leq q+j-1,
    \\
    (\phi \circ_j \psi) \circ_i \mu &=&  (\phi \circ_{i-q+1} \mu) \circ_{j} \psi &  \mbox{if} \ \ q+j \leq i \leq p+q-1,
    \\
\mu \circ_1     (\phi \circ_j \psi)  &=&  (\mu \circ_1 \phi) \circ_{j} \psi, &  
  \end{array}
\end{equation}
holds, along with
\begin{equation}
  \label{useless2}
\hspace*{-2.18cm}
  \begin{array}{rclll}
     (\phi \circ_{j-1}  \mu) \circ_j \psi &=&  \phi \circ_{j-1} (\mu \circ_2 \psi) &  \mbox{if} \ \ 2 \leq j \leq p,
    \\
    (\phi \circ_j \mu) \circ_j \psi &=&  \phi \circ_j (\mu \circ_1 \psi) &  \mbox{if} \ \ 1 \leq j \leq p,
    \\
    (\mu \circ_1 \phi) \circ_{p+1} \psi &=&  (\mu \circ_2 \psi) \circ_1 \phi. &  
  \end{array}
\end{equation}
Analogous relations can be written down for the unit element $e \in \cO(0)$.

These identities now do not look too exciting stated this way, but become more illuminating if recalling from \S\ref{pamukkale1} the cosimplicial structure of the sequence $\cO$ of $k$-modules, {\em i.e.},
\begin{equation}
  \label{sanf}
\gd_0 \phi = \mu \circ_2 \phi, \quad \gd_i \phi = \phi \circ_i \mu, \quad \gd_{p+1} \phi = \mu \circ_1 \phi, \quad \gs_j \phi  = \phi \circ_{j+1} e, 
  \end{equation}
where $1 \leq i \leq p$ and $0 \leq j \leq p-1$. Rewriting the relations \eqref{useless1}--\eqref{useless2} therefore in terms of the identities \eqref{sanf} reveals a compatibility between the cofaces and codegeneracies with the operadic composition, which is formalised (in the same order as the above relations) in the subsequent Definition \ref{compitompi}, and which will be the starting point for the following considerations. The upshot of these will be that {\em any} operad, not necessarily with multiplication, but nevertheless compatible in this sense with a given cosimplicial structure, induces a Gerstenhaber structure on cohomology, and in certain cases this structure is even that of a BV algebra. Hence, combining \eqref{sanf} with \eqref{useless1} \& \eqref{useless2}, let us define:

\begin{definition}
  \label{compitompi}
  Let $\cP$ be a (nonsymmetric) operad in $\kmod$, not necessarily with multiplication.
  \begin{enumerate}
\compactlist{99}
  \item
    The operad $\cP$ will be called {\em cosimplicial-compatible} if it is at the same time a cosimplicial
  object $(\cP^\bull, \gd_\bull, \gs_\bull)$ in $\kmod$ such that the
(partial, vertical) operadic composition is compatible with cofaces and codegeneracies in the following way: 
 for $x \in \cM(p)$ and $y \in \cM(q)$,
  \begin{eqnarray}
    \label{besser1}
 \gd_i (x \circ_j y) &=& 
\begin{cases}
  (\gd_i x) \circ_{j+1} y & \mbox{if} \  \,  0 \leq i \leq j-1, \\
x \circ_j (\gd_{i - j+1} y) & \mbox{if} \ \,  j  \leq i \leq j+q-1, \\
(\gd_{i-q+1} x) \circ_j y & \mbox{if} \ \, j+q \leq i \leq  p+q,
\end{cases}
  \end{eqnarray}
is supposed to hold for $1 \leq j \leq p$, along with
    \begin{eqnarray}
    \label{besser2}
 (\gd_i x) \circ_j y &=& 
 \begin{cases}
   x \circ_{j-1} (\gd_0 y) & \mbox{if} \  \,  i = j-1, \,\, 2 \leq j \leq p+1,
   \\
   x \circ_{j} (\gd_{q+1} y)   & \mbox{if} \ \,  i= j, \,\, 1 \leq j \leq p,
   \\
   (\gd_0 y) \circ_{1} x & \mbox{if} \  \,  i= j= p+1,
\end{cases}
\end{eqnarray}
  and, as far as the codegeneracies are concerned, 
    \begin{eqnarray}
    \label{besser3}
\gs_i (x \circ_j y) &=& 
\begin{cases}
  (\gs_i x) \circ_{j-1} y & \mbox{if} \  \,  0 \leq i \leq j-2,
  \\
  x \circ_j (\gs_{i - j +1} y) & \mbox{if} \ \,  j-1
  \leq i \leq j + q - 2, \\
(\gs_{i-q+1} x) \circ_{j} y & \mbox{if} \ \, j+q-1 \leq i \leq  p+q -2,
\end{cases}
    \end{eqnarray}
    is required to be true, where
$2 \leq j \leq p$ in the first line, $1 \leq j \leq p$ in the second, and     $1 \leq j \leq p-1$ in the third. 
  \item
    A cosimplicial-compatible operad $\cP$ is called {\em cocyclic-compatible} if
    it is a cocyclic object $(\cP^\bull, \gd_\bull, \gs_\bull, \tau)$ in $\kmod$ that is cosimplicial-compatible as above and, in addition,
    the partial operadic composition is compatible with the cocyclic operator, {\em i.e.},  if 
  \begin{eqnarray}
\label{unschoen}
    \tau( x \circ_j y) & = &
  \begin{cases}
    \tau y \circ_q \tau x & \mbox{if} \ j = 1, 
\\
    \tau x \circ_{j-1} y & \mbox{if} \ 2 \leq j \leq p,
\end{cases}
\end{eqnarray}
holds as well. 
\end{enumerate}
  \end{definition}

\begin{rem}
In case $y \in \cP(0)$, as before the middle relations in Eqs.~\eqref{besser1} \& \eqref{besser3} have to be ignored.
Together with $\tau^{n+1} = \id$ in degree $n$, the relations \eqref{unschoen} precisely mean that $(\cP, \tau)$ is a cyclic operad as defined (with an opposite convention) in \cite{GetKap:COACH}, but not necessarily a cyclic operad with multiplication as introduced in \cite{Men:BVAACCOHA}. Hence, a cocyclic-compatible operad $\cP$ in $\kmod$ is a cyclic operad with a cosimplicial structure $(\cP, \gd_\bull, \gs_\bull)$ that is compatible with the operadic composition in the sense of Eqs.~\eqref{besser1}--\eqref{besser3}.
\end{rem}

\subsection{Cosimplicial and cocyclic compatibility for the cyclic dual}
\label{mendelssohn}
With our main objective in mind, let us apply Definition \ref{compitompi} to the cosimplicial and cocyclic $k$-module $(\hat\cM, \gd_\bull, \gs_\bull, \tau)$ given by the cyclic dual of a cyclic opposite module $(\cM, t)$ over an operad with multiplication $(\cO, \mu, e)$ as described in Lemma \ref{Bop}. By means of Eqs.~\eqref{journuit}, the cosimplicial compatibility conditions \eqref{besser1}--\eqref{besser3}
can be rewritten in terms of the multiplication structure $(\mu, e)$ of $\cO$ acting on $(\cM, t)$.
%
 More precisely, with $x \in \hat\cM(p)$ and $y \in \hat\cM(q)$, condition \eqref{besser1} explicitly comes out as:
  \begin{eqnarray}
\label{opere1}
e \bullet_i (x \circ_j y) &=& 
\begin{cases}
  (e \bullet_i x) \circ_{j+1} y & \mbox{if} \  \,  0 \leq i \leq j-1, \\
x \circ_j (e \bullet_{i - j+1} y) & \mbox{if} \ \,  j  \leq i \leq j+q-1, \\
(e \bullet_{i-q+1} x) \circ_j y & \mbox{if} \ \, j+q \leq i \leq  p+q,
\end{cases}
 \end{eqnarray}
 for $1 \leq j \leq p$, whereas condition \eqref{besser2} reads 
\begin{eqnarray}
\label{opere2}
(e \bullet_i x) \circ_j y &=& 
\begin{cases}
  x \circ_{j-1} (e \bullet_0 y)  & \mbox{if} \ \, i = j-1, \,\, 2 \leq j \leq p+1,
  \\
  x \circ_{j} (e \bullet_{q+1} y)   & \mbox{if} \ \,  i= j, \,\, 1 \leq j \leq p,
  \\
(e \bullet_0 y) \circ_1 x & \mbox{if} \ \, i = j = p+1.
\end{cases}
\end{eqnarray}
Finally, condition 
 \eqref{besser3} can be stated by writing
    \begin{eqnarray}
\label{opere3}
\mu \bullet_i (x \circ_j y) &=& 
\begin{cases}
  (\mu \bullet_i x) \circ_{j-1} y & \mbox{if} \  \, 0 \leq i \leq j-2,
  \\
  x \circ_j (\mu \bullet_{i - j +1} y) & \mbox{if} \ \,  0 \leq j-1
  \leq i \leq j + q - 2, 
  \\
(\mu \bullet_{i-q+1} x) \circ_{j} y & \mbox{if} \ \, j+q-1 \leq i \leq  p+q -2,
\end{cases}
    \end{eqnarray}
    where
    $2 \leq j \leq p$ in line one, $1 \leq j \leq p$ in line two, and $1 \leq j \leq p-1$ in line three.
Again, the middle lines in Eqs.\ \eqref{opere1} \& \eqref{opere3} are not present in case $q = 0$, that is, for $y$ a zero cochain.
    From Eqs.~\eqref{opere1}--\eqref{opere2}, we can derive a couple of identities, which will be stated in a partially redundant manner but nevertheless quite convenient to have at hand in explicit computations: in~fact, 
\begin{eqnarray}
\label{opere4}
 (e \bullet_i x) \circ_j y &=& 
 \begin{cases}
   (e \bullet_{q+1} y) \circ_{q+1} x & \mbox{if} \  \,  i=0, \, j=1,
   \\
   e \bullet_i (x \circ_{j-1} y) & \mbox{if} \  \,  0 \leq i \leq j-2, \, 2 \leq j \leq p+1,
   \\
   x \circ_{j-1} (e \bullet_0  y) & \mbox{if} \  \,  i = j-1, \, 2 \leq j \leq p+1,
   \\
   x \circ_{j} (e \bullet_{q+1} y)   & \mbox{if} \ \,  i= j \neq p+1,
   \\
(e \bullet_0 y) \circ_{1} x & \mbox{if} \  \,  i= j= p+1,
   \\
e \bullet_{i+q-1} (x \circ_j y) & \mbox{if} \ \, j+1 \leq i \leq  p+1,
\end{cases}
\end{eqnarray}
along with
\begin{eqnarray}
\label{opere5}
x \circ_j (e \bullet_i y)  &=& 
\begin{cases}
  (e \bullet_j x) \circ_{j+1} y & \mbox{if} \  \,  i=0, 
  \\
  e \bullet_{i+j-1}  (x \circ_{j} y)  & \mbox{if} \ \,  1 \leq i \leq q,
  \\
(e \bullet_j x) \circ_j y & \mbox{if} \ \, i = q+1, 
\end{cases}
\end{eqnarray}
for $1 \leq j \leq p$ everywhere.
Moreover, Eq.~\eqref{opere3} can be restated as:
    \begin{equation}
  \begin{array}{rcll}
\label{opere6}
x \circ_j (\mu \bullet_i y) &=& \mu \bullet_{i+j-1} (x \circ_j y), \ 
\quad 1 \leq i \leq q-1, \ \  1 \leq j \leq p,
\\[2mm]
(\mu \bullet_i x) \circ_j y &=& 
\begin{cases}
  \mu \bullet_i (x  \circ_{j+1} y) & \mbox{if} \  \,  0 \leq i \leq j-1, \ \ 1 \leq j \leq p-1,
  \\
\mu \bullet_{i + q -1} ( x \circ_{j} y) & \mbox{if} \ \,  1 \leq j  \leq i \leq p -1.
\end{cases}
  \end{array}
  \end{equation}
    To conclude, the cyclic operad condition \eqref{unschoen} with respect to the cyclic operator $\tau = t^{-1}$ can be rewritten as
      \begin{eqnarray}
\label{unschoen2}
    t( x \circ_j y) & = &
  \begin{cases}
t x \circ_{j+1} y & \mbox{if} \ 1 \leq j \leq p-1,
    \\
        t y \circ_1 t x & \mbox{if} \ j = p, 
     \end{cases}
      \end{eqnarray}
      in terms of the original cyclic operator $t$ on $\cM$.

      \subsection{Graphical representation}
      In contrast to the more general relations in Definition \ref{compitompi}, the identities \eqref{opere1}--\eqref{opere6} can be pictured in quite an instructive manner, which will also explain the term {\em tree reflexion}. As seen in Figure~\ref{tratto} on page~\pageref{bretzel}, an opposite module $\cM$ over an operad $\cO$ could be depicted as upside-down trees to which (normally oriented) trees, as elements in the operad, are plugged from below, in order to display the opposite action.

      If the cyclic dual $\hat\cM$ is itself an operad, one might think of mirroring this picture along a horizontal line so that an element in the opposite module becomes a tree with, in turn, the operad acting from above by upside-down trees. This, in view of what was said in the introduction, is justified by the idea that for the {\em full picture} one should rather pass to the cyclic duals of {\em both} $\cM$ and $\cO$.

      Applying such a horizontal reflexion, we can illustrate, for example, the first line in the relations \eqref{opere3}: 

      \begin{center}
        \scalebox{0.5}{
          
          \begin{tikzpicture}
	    \node [style= ] (1) at (0.5, -.5) {\Huge $=$};
            \node [fill=white, draw=none, shape=rectangle] (2) at (-1.5, -.5) {\huge $p$};
	    \node [fill=white, draw=none, shape=rectangle] (3) at (-3.5, -.5) {\huge $j$};
	   \node [fill=white, draw=none, shape=rectangle] (4) at (-6, -.5) {\huge $i+1$};
	\node [fill=white, draw=none, shape=rectangle] (5) at (-8, -.5) {\huge $i$};
	\node [fill=white, draw=none, shape=rectangle] (6) at (-10, -.5) {\huge $1$};
	\node [fill=white, draw=none, shape=rectangle] (7) at (-9, -.5) {\Large $\cdots$};
	\node [ ] (8) at (-4.75, -.5) {\Large $\cdots$};
	\node [ ] (9) at (-2.5, -.5) {\Large $\cdots$};
	\node [fill=white, draw=none, shape=rectangle] (10) at (-6, -3) {\Huge $x$};
	\node [ ] (11) at (-6, -4) {};
	\node [fill=white, draw=none, shape=rectangle] (12) at (-3.5, 1) {\Huge $y$};
	\node [ ] (14) at (-3.5, 3) {\Large $\cdots$};
	\node [fill=white, draw=none, shape=rectangle] (17) at (-8, 3) {};
	\node [fill=white, draw=none, shape=rectangle] (18) at (-6, 3) {};
	\node [ ] (20) at (-7, 5.5) {};
	\node [fill=white, draw=none, shape=rectangle] (21) at (-7, 4.25) {\Huge $\mu$};
	\node [fill=white, draw=none, shape=rectangle] (22) at (-4.75, 3) {\huge $1$};
	\node [fill=white, draw=none, shape=rectangle] (23) at (-2.25, 3) {\huge $q$};
	\node [ ] (24) at (-10, 3) {};
	\node [ ] (25) at (-10, -4) {};
	\node [ ] (26) at (-1.25, 3) {};
	\node [ ] (27) at (-1.25, -4) {};
	\node [fill=white, draw=none, shape=rectangle] (29) at (12, -0.5) {\huge $p$};
	\node [fill=white, draw=none, shape=rectangle] (30) at (10, -0.5) {\huge $j-1$};
	\node [fill=white, draw=none, shape=rectangle] (31) at (6.5, -0.5) {\huge $i+1$};
	\node [fill=white, draw=none, shape=rectangle] (32) at (4.5, -0.5) {\huge $i$};
	\node [fill=white, draw=none, shape=rectangle] (33) at (2.5, -0.5) {\huge $1$};
	\node [ ] (34) at (3.5, -0.5) {\Large $\cdots$};
	\node [ ] (35) at (7.6, -0.5) {\Large $\cdots$};
	\node [ ] (36) at (11.25, -0.5) {\Large $\cdots$};
	\node [fill=white, draw=none, shape=rectangle] (37) at (6.5, -3) {\Huge $x$};
	\node [ ] (38) at (6.5, -4) {};
	\node [fill=white, draw=none, shape=rectangle] (39) at (10, 1) {\Huge $y$};
	\node [ ] (41) at (10, 3) {\Large $\cdots$};
	\node [ ] (44) at (5.5, 3) {};
	\node [fill=white, draw=none, shape=rectangle] (45) at (5.5, 1) {\Huge $\mu$};
	\node [fill=white, draw=none, shape=rectangle] (46) at (8.75, 3) {\huge $1$};
	\node [fill=white, draw=none, shape=rectangle] (47) at (11.25, 3) {\huge $q$};
	\node [ ] (48) at (2.5, 3) {};
	\node [ ] (49) at (2.5, -4) {};
	\node [ ] (50) at (7.5, 3) {};
	\node [ ] (51) at (7.5, -4) {};
%
	\draw [-, dashed, line width=1.2pt] (7) to (10);
	\draw [-, dashed, line width=1.2pt] (8) to (10);
	\draw [-, dashed, line width=1.2pt] (9) to (10);
	\draw [-, dashed, line width=1.2pt] (14) to (12);
	\draw [-, fill=white, draw=black, line width=1.2pt] (12) to (3);
	\draw [-, fill=white, draw=black, line width=1.2pt] (3) to (10);
	\draw [-, fill=white, draw=black, line width=1.2pt] (10) to (11);
	\draw [-, fill=white, draw=black, line width=1.2pt] (2) to (10);
	\draw [-, fill=white, draw=black, line width=1.2pt] (4) to (10);
	\draw [-, fill=white, draw=black, line width=1.2pt] (5) to (10);
	\draw [-, fill=white, draw=black, line width=1.2pt] (6) to (10);
	\draw [-, fill=white, draw=black, line width=1.2pt] (17) to (5);
	\draw [-, fill=white, draw=black, line width=1.2pt] (18) to (4);
	\draw [-, fill=white, draw=black, line width=1.2pt] (20) to (21);
	\draw [-, fill=white, draw=black, line width=1.2pt] (21) to (17);
	\draw [-, fill=white, draw=black, line width=1.2pt] (21) to (18);
	\draw [-, fill=white, draw=black, line width=1.2pt] (23) to (12);
	\draw [-, fill=white, draw=black, line width=1.2pt] (22) to (12);
	\draw [-, draw=black, line width=3pt, bend right=15] (24.center) to (25.center);
	\draw [-, draw=black, line width=3pt, bend left=15] (26.center) to (27.center);
	\draw [-, dashed, line width=1.2pt] (34) to (37);
	\draw [-, dashed, line width=1.2pt] (35) to (37);
	\draw [-, dashed, line width=1.2pt] (36) to (37);
	\draw [-, dashed, line width=1.2pt] (41) to (39);
	\draw [-, fill=white, draw=black, line width=1.2pt] (39) to (30);
	\draw [-, fill=white, draw=black, line width=1.2pt] (30) to (37);
	\draw [-, fill=white, draw=black, line width=1.2pt] (37) to (38.center);
	\draw [-, fill=white, draw=black, line width=1.2pt] (29) to (37);
	\draw [-, fill=white, draw=black, line width=1.2pt] (31) to (37);
	\draw [-, fill=white, draw=black, line width=1.2pt] (32) to (37);
	\draw [-, fill=white, draw=black, line width=1.2pt] (33) to (37);
	\draw [-, fill=white, draw=black, line width=1.2pt] (45) to (32);
	\draw [-, fill=white, draw=black, line width=1.2pt] (45) to (31);
	\draw [-, fill=white, draw=black, line width=1.2pt] (44.center) to (45);
	\draw [-, fill=white, draw=black, line width=1.2pt] (47) to (39);
	\draw [-, fill=white, draw=black, line width=1.2pt] (46) to (39);
	\draw [-, draw=black, line width=3pt, bend right=15] (48.center) to (49.center);
	\draw [-, draw=black, line width=3pt, bend left=15] (50.center) to (51.center);
\end{tikzpicture}
        }
        \medskip
        
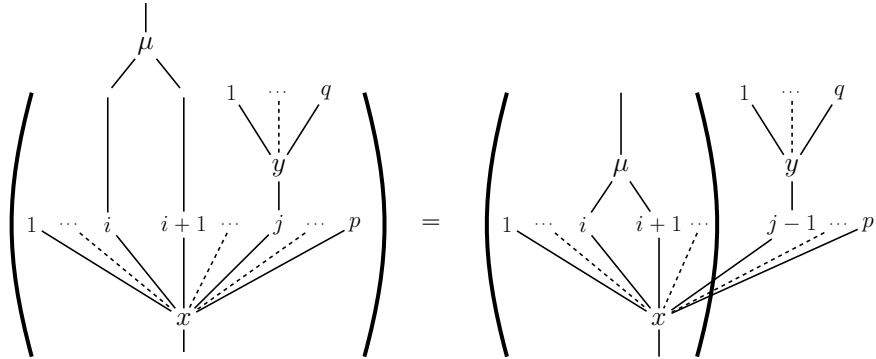
\captionof{figure}{The relation $\mu \bullet_i ( x \circ_j y) = (\mu \bullet_i x) \circ_{j-1} y$.}
        \end{center}
     This figure motivates to call the depicted relation a {\em parallel associativity}, while an analogous picture for the middle line in \eqref{opere3} would be rightly called a {\em sequential associativity}, as is common for the operadic composition. However, it is more interesting to represent the situation at the borders, {\em i.e.}, for the cases $i = j-1$ and $i = j+q-1$: the first one takes into account that the operation $\mu \bullet_0 -$ cannot be depicted by an (upside-down) binary tree anymore but rather by a (normally oriented) truncated tree (which is {\em not} the unit element $e$ in the multiplication structure of the operad). More precisely: 
     \begin{center}
        \scalebox{0.5}{
          
          \begin{tikzpicture}
	    \node [style= ] (1) at (0.5, -.5) {\Huge $=$};
            \node [fill=white, draw=none, shape=rectangle] (2) at (-1.5, -.5) {\huge $p$};
	    \node [fill=white, draw=none, shape=rectangle] (3) at (-3.5, -.5) {\huge $j$};
	   \node [fill=white, draw=none, shape=rectangle] (4) at (-6, -.5) {\huge $j-1$};
	\node [fill=white, draw=none, shape=rectangle] (5) at (-7.5, -.5) {\Large $\cdots$};
	\node [fill=white, draw=none, shape=rectangle] (6) at (-8.75, -.5) {\huge $1$};
	\node [ ] (9) at (-2.5, -.5) {\Large $\cdots$};
	\node [fill=white, draw=none, shape=rectangle] (10) at (-6, -3) {\Huge $x$};
	\node [ ] (11) at (-6, -4) {};
	\node [fill=white, draw=none, shape=rectangle] (12) at (-3.5, 1) {\Huge $y$};
	\node [ ] (14) at (-3.5, 3) {\Large $\cdots$};
	\node [fill=white, draw=none, shape=rectangle] (17) at (-8, 3) {};
	\node [fill=white, draw=none, shape=rectangle] (18) at (-6, 3) {};
	\node [ ] (20) at (-5.5, 5.5) {};
	\node [fill=white, draw=none, shape=rectangle] (21) at (-5.5, 4.25) {\Huge $\mu$};
	\node [fill=white, draw=none, shape=rectangle] (22) at (-4.75, 3) {\huge $1$};
	\node [fill=white, draw=none, shape=rectangle] (23) at (-2.25, 3) {\huge $q$};
	\node [ ] (24) at (-9, 3) {};
	\node [ ] (25) at (-9, -4) {};
	\node [ ] (26) at (-1.25, 3) {};
	\node [ ] (27) at (-1.25, -4) {};
	\node [fill=white, draw=none, shape=rectangle] (29) at (11.5, -0.5) {\huge $p$};
	\node [fill=white, draw=none, shape=rectangle] (30) at (8.5, -0.5) {\huge $j$};
	\node [fill=white, draw=none, shape=rectangle] (31) at (5, -0.5) {\huge $j-1$};
	\node [fill=white, draw=none, shape=rectangle] (32) at (3, -0.5) {\Large $\cdots$};
	\node [fill=white, draw=none, shape=rectangle] (33) at (1.5, -0.5) {\huge $1$};
	\node [ ] (36) at (9.75, -0.5) {\Large $\cdots$};
	\node [fill=white, draw=none, shape=rectangle] (37) at (5, -3) {\Huge $x$};
	\node [ ] (38) at (5, -4) {};
	\node [fill=white, draw=none, shape=rectangle] (39) at (8.5, 1) {\Huge $y$};
	\node [ ] (41) at (8.5, 3) {\Large $\cdots$};
	\node [fill=white, draw=none, shape=rectangle] (46) at (7.25, 3) {\huge $1$};
	\node [fill=white, draw=none, shape=rectangle] (47) at (9.75, 3) {\huge $q$};
	\node [ ] (48) at (6.75, 4.75) {};
	\node [ ] (49) at (6.75, -4) {};
	\node [ ] (50) at (10, 4.75) {};
	\node [ ] (51) at (10, -4) {};
\node [fill=white, draw=none, shape=rectangle] (52) at (7.25, 4.75) {\huge };
	\draw [-, dashed, line width=1.2pt] (9) to (10);
	\draw [-, dashed, line width=1.2pt] (14) to (12);
	\draw [-, fill=white, draw=black, line width=1.2pt] (12) to (3);
	\draw [-, fill=white, draw=black, line width=1.2pt] (3) to (10);
	\draw [-, fill=white, draw=black, line width=1.2pt] (10) to (11.center);
	\draw [-, fill=white, draw=black, line width=1.2pt] (2) to (10);
	\draw [-, fill=white, draw=black, line width=1.2pt] (4) to (10);
	\draw [dashed, fill=white, draw=black, line width=1.2pt] (5) to (10);
	\draw [-, fill=white, draw=black, line width=1.2pt] (6) to (10);
	\draw [-, fill=white, draw=black, line width=1.2pt] (18) to (4);
	\draw [-, fill=white, draw=black, line width=1.2pt] (20.center) to (21);
	\draw [-, fill=white, draw=black, line width=1.2pt] (21) to (22);
	\draw [-, fill=white, draw=black, line width=1.2pt] (21) to (18);
	\draw [-, fill=white, draw=black, line width=1.2pt] (23) to (12);
	\draw [-, fill=white, draw=black, line width=1.2pt] (22) to (12);
	\draw [-, draw=black, line width=3pt, bend right=15] (24.center) to (25.center);
	\draw [-, draw=black, line width=3pt, bend left=15] (26.center) to (27.center);
	\draw [-, dashed, line width=1.2pt] (36) to (37);
	\draw [-, dashed, line width=1.2pt] (41) to (39);
	\draw [-, fill=white, draw=black, line width=1.2pt] (39) to (30);
	\draw [-, fill=white, draw=black, line width=1.2pt] (30) to (37);
	\draw [-, fill=white, draw=black, line width=1.2pt] (37) to (38);
	\draw [-, fill=white, draw=black, line width=1.2pt] (29) to (37);
	\draw [-, fill=white, draw=black, line width=1.2pt] (31) to (37);
	\draw [dashed, fill=white, draw=black, line width=1.2pt] (32) to (37);
	\draw [-, fill=white, draw=black, line width=1.2pt] (33) to (37);
	\draw [-, fill=white, draw=black, line width=1.2pt] (47) to (39);
	\draw [-, fill=white, draw=black, line width=1.2pt] (46) to (39);
	\draw [-, draw=black, line width=3pt, bend right=15] (48.center) to (49.center);
	\draw [-, draw=black, line width=3pt, bend left=15] (50.center) to (51.center);
        	\draw [Bracket-, fill=white, draw=black, line width=1.2pt] (52) to (46);
\end{tikzpicture}
        }
        \medskip
        
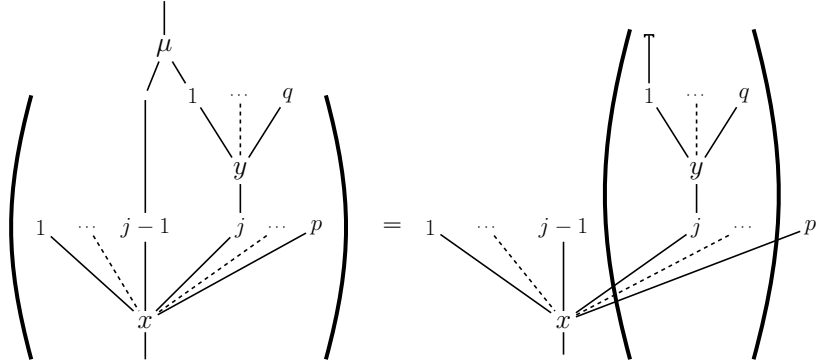
\captionof{figure}{The relation $\mu \bullet_{j-1} ( x \circ_j y) = x \circ_j (\mu \bullet_0 y)$.}
        \end{center}
%
%
    \noindent The second border case mentioned above, where $i = j+q-1$, would look like:
     \smallskip
     \begin{center}
        \scalebox{0.5}{
          \begin{tikzpicture}
	    \node [style= ] (1) at (0.2, -.5) {\Huge $=$};
            \node [fill=white, draw=none, shape=rectangle] (2) at (-1.5, -.5) {\huge $p$};
	    \node [fill=white, draw=none, shape=rectangle] (3) at (-3.5, -.5) {\huge $j+1$};
	   \node [fill=white, draw=none, shape=rectangle] (4) at (-6, -.5) {\huge $j$};
	\node [fill=white, draw=none, shape=rectangle] (5) at (-7.5, -.5) {\Large $\cdots$};
	\node [fill=white, draw=none, shape=rectangle] (6) at (-8.75, -.5) {\huge $1$};
	\node [ ] (9) at (-2.25, -.5) {\Large $\cdots$};
	\node [fill=white, draw=none, shape=rectangle] (10) at (-6, -3) {\Huge $x$};
	\node [ ] (11) at (-6, -4) {};
	\node [fill=white, draw=none, shape=rectangle] (12) at (-6, 1) {\Huge $y$};
	\node [ ] (14) at (-6, 3) {\Large $\cdots$};
	\node [fill=white, draw=none, shape=rectangle] (17) at (-8, 3) {};
	\node [fill=white, draw=none, shape=rectangle] (18) at (-3.5, 3) {};
	\node [ ] (20) at (-4.1, 5.5) {};
	\node [fill=white, draw=none, shape=rectangle] (21) at (-4.1, 4.25) {\Huge $\mu$};
	\node [fill=white, draw=none, shape=rectangle] (22) at (-7.25, 3) {\huge $1$};
	\node [fill=white, draw=none, shape=rectangle] (23) at (-4.75, 3) {\huge $q$};
	\node [ ] (24) at (-9, 3) {};
	\node [ ] (25) at (-9, -4) {};
	\node [ ] (26) at (-1.25, 3) {};
	\node [ ] (27) at (-1.25, -4) {};
	\node [fill=white, draw=none, shape=rectangle] (29) at (8, -0.5) {\huge $p$};
	\node [fill=white, draw=none, shape=rectangle] (30) at (5.5, -0.5) {\huge $j+1$};
	\node [fill=white, draw=none, shape=rectangle] (31) at (4, -0.5) {\huge $j$};
	\node [fill=white, draw=none, shape=rectangle] (32) at (2.5, -0.5) {\huge \Large $\cdots$};
	\node [fill=white, draw=none, shape=rectangle] (33) at (1.5, -0.5) {\huge $1$};
	\node [ ] (36) at (6.75, -0.5) {\Large $\cdots$};
	\node [fill=white, draw=none, shape=rectangle] (37) at (4, -3) {\Huge $x$};
	\node [ ] (38) at (4, -4) {};
	\node [fill=white, draw=none, shape=rectangle] (39) at (4.75, 3) {\Huge $y$};
	\node [ ] (41) at (4.75, 4.5) {\Large $\cdots$};
	\node [fill=white, draw=none, shape=rectangle] (45) at (4.75, 1) {\Huge $\mu$};
	\node [fill=white, draw=none, shape=rectangle] (46) at (3.75, 4.75) {\huge $1$};
	\node [fill=white, draw=none, shape=rectangle] (47) at (5.75, 4.75) {\huge $q$};
	\node [ ] (48) at (1.5, 2) {};
	\node [ ] (49) at (1.5, -4) {};
	\node [ ] (50) at (8.25, 2) {};
	\node [ ] (51) at (8.25, -4) {};
\node [fill=white, draw=none, shape=rectangle] (52) at (7.25, 4.75) {\huge };
 	\draw [-, dashed, line width=1.2pt] (9) to (10);
	\draw [-, dashed, line width=1.2pt] (14) to (12);
	\draw [-, fill=white, draw=black, line width=1.2pt] (12) to (4);
	\draw [-, fill=white, draw=black, line width=1.2pt] (3) to (10);
	\draw [-, fill=white, draw=black, line width=1.2pt] (10) to (11);
	\draw [-, fill=white, draw=black, line width=1.2pt] (2) to (10);
	\draw [-, fill=white, draw=black, line width=1.2pt] (4) to (10);
	\draw [dashed, fill=white, draw=black, line width=1.2pt] (5) to (10);
	\draw [-, fill=white, draw=black, line width=1.2pt] (6) to (10);
	\draw [-, fill=white, draw=black, line width=1.2pt] (18) to (3);
	\draw [-, fill=white, draw=black, line width=1.2pt] (20) to (21);
	\draw [-, fill=white, draw=black, line width=1.2pt] (21) to (23);
	\draw [-, fill=white, draw=black, line width=1.2pt] (21) to (18);
	\draw [-, fill=white, draw=black, line width=1.2pt] (23) to (12);
	\draw [-, fill=white, draw=black, line width=1.2pt] (22) to (12);
	\draw [-, draw=black, line width=3pt, bend right=15] (24.center) to (25.center);
	\draw [-, draw=black, line width=3pt, bend left=15] (26.center) to (27.center);
	\draw [-, dashed, line width=1.2pt] (36) to (37);
	\draw [-, dashed, line width=1.2pt] (41) to (39);
	\draw [-, fill=white, draw=black, line width=1.2pt] (39) to (45);
	\draw [-, fill=white, draw=black, line width=1.2pt] (30) to (37);
	\draw [-, fill=white, draw=black, line width=1.2pt] (37) to (38);
	\draw [-, fill=white, draw=black, line width=1.2pt] (29) to (37);
	\draw [-, fill=white, draw=black, line width=1.2pt] (31) to (37);
	\draw [dashed, fill=white, draw=black, line width=1.2pt] (32) to (37);
	\draw [-, fill=white, draw=black, line width=1.2pt] (33) to (37);
	\draw [-, fill=white, draw=black, line width=1.2pt] (47) to (39);
	\draw [-, fill=white, draw=black, line width=1.2pt] (46) to (39);
        \draw [-, fill=white, draw=black, line width=1.2pt] (45) to (31);
        \draw [-, fill=white, draw=black, line width=1.2pt] (45) to (30);
	\draw [-, draw=black, line width=3pt, bend right=15] (48.center) to (49.center);
	\draw [-, draw=black, line width=3pt, bend left=15] (50.center) to (51.center);
\end{tikzpicture}
        }
        \medskip
        
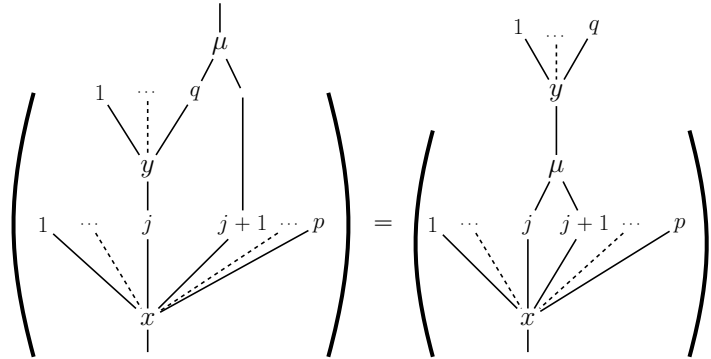
\captionof{figure}{The relation $\mu \bullet_{j+q-1} ( x \circ_j y) = (\mu \bullet_j x) \circ_j y$.}
        \end{center}
     Observe that on the right hand side there is no further associativity that allows to, {\em e.g.}, isolate the composition of the first tree from above to the upside-down one given by $\mu$, as this would rather involve the notion of PROPs.

     Analogous diagrams can be drawn with respect to the unit element $e \in \cO(0)$, which, in this spirit, acts from above as a truncated upside-down tree.
     For ex\-ample, the first line of the relations in \eqref{opere1} looks as expected:
     \begin{center}
       \scalebox{0.6}
                {
\begin{tikzpicture}
		\node [fill=white, draw=none, shape=rectangle] (0) at (-0.1, 0) {\huge $=$};
		\node [fill=white, draw=none, shape=rectangle] (1) at (-2, 0) {$\cdots$};
		\node [fill=white, draw=none, shape=rectangle] (2) at (-2.8, 0) {\LARGE $j$};
		\node [fill=white, draw=none, shape=rectangle] (3) at (-4.75, 0) {$\cdots$};
		\node [fill=white, draw=none, shape=rectangle] (4) at (-5, -2) {\huge $x$};
		\node [fill=white, draw=none, shape=rectangle] (6) at (-7, 0) {\LARGE $i-1$};
		\node [fill=white, draw=none, shape=rectangle] (7) at (-8.2, 0) {$\cdots$};
		\node [fill=white, draw=none, shape=rectangle] (8) at (-2.8, 1.3) {\huge $y$};
		\node [fill=white, draw=none, shape=rectangle] (9) at (-3.8, 2.5) {\LARGE $1$};
		\node [fill=white, draw=none, shape=rectangle] (10) at (-2.8, 2.5) {$\cdots$};
		\node [fill=white, draw=none, shape=rectangle] (11) at (-1.8, 2.5) {\LARGE $q$};
		\node [fill=white, draw=none, shape=rectangle] (12) at (-5, -3) {};
		\node [fill=white, draw=none, shape=rectangle] (13) at (-6, 3.8) {\LARGE $i$};
		\node [fill=white, draw=none, shape=rectangle] (14) at (-6, 2.7) {\huge $e$};
		\node [fill=white, draw=none, shape=rectangle] (16) at (2, 0) {$\cdots$};
		\node [fill=white, draw=none, shape=rectangle] (17) at (3, 0) {\LARGE $i-1$};
		\node [fill=white, draw=none, shape=rectangle] (18) at (3.9, 0.05) {\LARGE $i$};
		\node [fill=white, draw=none, shape=rectangle] (19) at (5, -2) {\huge $x$};
		\node [fill=white, draw=none, shape=rectangle] (20) at (5, -3) {};
		\node [fill=white, draw=none, shape=rectangle] (21) at (5.5, 0) {$\cdots$};
		\node [fill=white, draw=none, shape=rectangle] (22) at (7, 0) {\LARGE $j+1$};
		\node [fill=white, draw=none, shape=rectangle] (23) at (8.25, 0) {$\cdots$};
		\node [fill=white, draw=none, shape=rectangle] (24) at (7, 1) {\huge $y$};
		\node [fill=white, draw=none, shape=rectangle] (25) at (6, 2) {\LARGE $1$};
		\node [fill=white, draw=none, shape=rectangle] (26) at (7, 2) {$\cdots$};
		\node [fill=white, draw=none, shape=rectangle] (27) at (8, 2) {\LARGE $q$};
		\node [fill=white, draw=none, shape=rectangle] (28) at (4.65, -1.1) {\huge $e$};
		\node [fill=white, draw=none, shape=rectangle] (29) at (-1.1, 0) {\LARGE $p$};
		\node [fill=white, draw=none, shape=rectangle] (30) at (9, 0) {\LARGE $p$};
		\node [fill=white, draw=none, shape=rectangle] (31) at (1, 0) {\LARGE $1$};
		\node [fill=white, draw=none, shape=rectangle] (32) at (-9, 0) {\LARGE $1$};
             \node [] (33) at (-8.9,3) {};
                \node [] (34) at (-8.9,-3) {};
                \node [] (35) at (-1.1,3) {};
                \node [] (36) at (-1.1,-3) {};
     \node [] (37) at (.7,.5) {};
                \node [] (38) at (.7,-3) {};
                \node [] (39) at (5.9,.5) {};
                \node [] (40) at (5.9,-3) {};

		\draw [-Bracket, line width=1.2pt](13) to (14);
		\draw [dashed, line width=1.2pt](7) to (4);
		\draw [-, line width=1.2pt](4) to (12);
		\draw [-, line width=1.2pt](6) to (4);
		\draw [dashed, line width=1.2pt](3) to (4);
		\draw [-, line width=1.2pt](2) to (4);
		\draw [dashed, line width=1.2pt](1) to (4);
		\draw [-, line width=1.2pt](8) to (2);
		\draw [-, line width=1.2pt](9) to (8);
		\draw [dashed, line width=1.2pt](10) to (8);
		\draw [-, line width=1.2pt](11) to (8);
		\draw [-, line width=1.2pt](29) to (4);
		\draw [-, line width=1.2pt](19) to (20);
		\draw [-Bracket, line width=1.2pt](18) to (28);
		\draw [-, line width=1.2pt](17) to (19);
		\draw [dashed, line width=1.2pt](16) to (19);
		\draw [dashed, line width=1.2pt](21) to (19);
		\draw [-, line width=1.2pt](22) to (19);
		\draw [dashed, line width=1.2pt](23) to (19);
		\draw [-, line width=1.2pt](30) to (19);
		\draw [-, line width=1.2pt](24) to (22);
		\draw [-, line width=1.2pt][in=135, out=-45] (25) to (24);
		\draw [dashed, line width=1.2pt](26) to (24);
		\draw [-, line width=1.2pt](27) to (24);
		\draw [-, line width=1.2pt](32) to (4);
		\draw [-, line width=1.2pt](31) to (19);
              \draw [-, draw=black, line width=3pt, bend right=15] (33.center) to (34.center);
	      \draw [-, draw=black, line width=3pt, bend left=15] (35.center) to (36.center);
                            \draw [-, draw=black, line width=3pt, bend right=15] (37.center) to (38.center);
	        \draw [-, draw=black, line width=3pt, bend left=15] (39.center) to (40.center);
\end{tikzpicture}         
        }
        \medskip
        
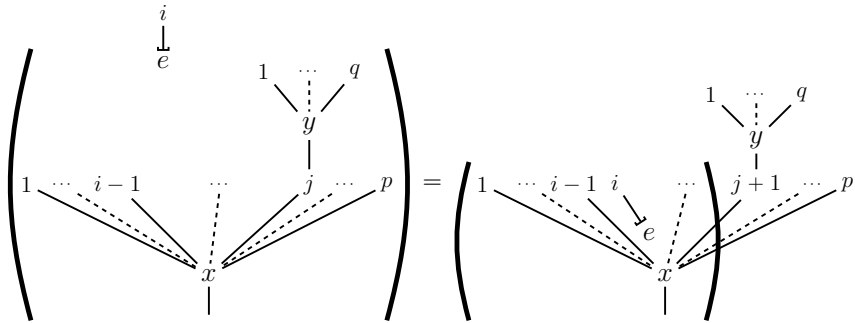
\captionof{figure}{The relation $e \bullet_{i} (x \circ_{j} y) = (e \bullet_i x) \circ_{j+1} y$.}
        \end{center}
    More curious are those cases in which the extra degeneracy $e \bullet_0 -$ appears since this involves the cyclic operator $t$, and hence a bending similar to Figure \ref{rel}. For example, the first line of \eqref{opere2} graphically comes out as:
  \begin{center}
       \scalebox{0.75}
                {
\begin{tikzpicture}
		\node [fill=white, draw=none, shape=rectangle] (0) at (0.25, -3) {\LARGE $=$};
		\node [fill=white, draw=none, shape=rectangle] (1) at (-4.5, -3.55) {\Large $j-1$};
		\node [fill=white, draw=none, shape=rectangle] (2) at (-3.5, -1.5) {\LARGE $y$};
		\node [fill=white, draw=none, shape=rectangle] (3) at (-4.5, 0) {\Large $1$};
		\node [fill=white, draw=none, shape=rectangle] (4) at (-3.5, 0) {$\cdots$};
		\node [fill=white, draw=none, shape=rectangle] (5) at (-2.5, 0) {\Large $q$};
		\node [fill=white, draw=none, shape=rectangle] (11) at (2, 0) {\Large $0$};
		\node [fill=white, draw=none, shape=rectangle] (12) at (4, 0) {$\cdots$};
		\node [fill=white, draw=none, shape=rectangle] (17) at (5, 0) {\Large $q$};
		\node [fill=white, draw=none, shape=rectangle] (18) at (4, -1.5) {\LARGE $y$};
		\node [fill=white, draw=none, shape=rectangle] (19) at (-3.75, -4.63) {\LARGE $e$};
		\node [fill=white, draw=none, shape=rectangle] (20) at (6, 0) {\Large $q+1$};
		\node [fill=white, draw=none, shape=rectangle] (21) at (4.75, -1) {\LARGE $e$};
		\node [fill=white, draw=none, shape=rectangle] (22) at (4.75, -2.25) {};
		\node [fill=white, draw=none, shape=rectangle] (23) at (-4, -3.25) {};
		\node [fill=white, draw=none, shape=rectangle] (27) at (3, 0) {\Large $1$};
		\node [fill=white, draw=none, shape=rectangle] (28) at (4, -3.25) {};
                \node [] (t) at (8.1, -2.5) {\LARGE $t$};

                \node [fill=white, draw=none, shape=rectangle] (13) at (4, -5.25) {\LARGE $x$};
                     \node [fill=white, draw=none, shape=rectangle] (13a) at (4, -6) {\LARGE };
		\node [fill=white, draw=none, shape=rectangle] (14) at (1.25, -3.5) {\Large $1$};
	        \node [fill=white, draw=none, shape=rectangle] (15) at (2, -3.5) {$\cdots$};
                	\node [fill=white, draw=none, shape=rectangle] (15a) at (4, -3.5) {$j-1$};
	\node [fill=white, draw=none, shape=rectangle] (15b) at (6, -3.5) {$\cdots$};
	\node [fill=white, draw=none, shape=rectangle] (16) at (6.75, -3.5) {\Large $p$};

                        \node [fill=white, draw=none, shape=rectangle] (24) at (-3.5, -5.25) {\LARGE $x$};
                     \node [fill=white, draw=none, shape=rectangle] (25) at (-3.5, -6) {\LARGE };
		\node [fill=white, draw=none, shape=rectangle] (26) at (-6.25, -3.5) {\Large $1$};
	        \node [fill=white, draw=none, shape=rectangle] (24a) at (-5.5, -3.5) {$\cdots$};
                	\node [fill=white, draw=none, shape=rectangle] (25a) at (-3.5, -3.5) {$j$};
	\node [fill=white, draw=none, shape=rectangle] (26a) at (-1.5, -3.5) {$\cdots$};
	\node [fill=white, draw=none, shape=rectangle] (6) at (-.75, -3.5) {\Large $p$};

                \node [] (29) at (-6.6, -3.25) {};
                \node [] (30) at (-6.6, -6) {};
                \node [] (31) at (-.4, -3.25) {};
                \node [] (32) at (-.4, -6) {};
                \node [] (33) at (1.5, .5) {};
                \node [] (34) at (1.5, -3.25) {};
                \node [] (35) at (7.75, .5) {};
                \node [] (36) at (7.75, -3.25) {};

		\draw [-, line width=1.2pt]  (3) to (2);
		\draw [dashed, line width=1.2pt]  (4) to (2);
		\draw [-, line width=1.2pt]  (5) to (2);
		\draw [-Bracket, line width=1.2pt]  (1) to (19);
		\draw [-, line width=1.2pt]  (6) to (24);
		\draw [dashed, line width=1.2pt]  (26) to (24);
		\draw [-, line width=1.2pt]  (25.center) to (24);
		\draw [-, line width=1.2pt]  (25a) to (2);
		\draw [-, line width=1.2pt]  (14) to (13);
        	\draw [dashed, line width=1.2pt]  (15) to (13);
                \draw [-, line width=1.2pt]  (15a) to (13);
                \draw [-, line width=1.2pt]  (13a.center) to (13);
                \draw [dashed, line width=1.2pt]  (15b) to (13);
		\draw [-, line width=1.2pt]  (27) to (18);
		\draw [dashed, line width=1.2pt]  (12) to (18);
		\draw [-, line width=1.2pt]  (17) to (18);
		\draw [-Bracket, line width=1.2pt]  (20) to (21);
                \draw [-, line width=1.2pt]  (24) to (26);
                \draw [dashed, line width=1.2pt]  (24) to (24a);
                \draw [dashed, line width=1.2pt]  (24) to (26a);
                \draw [-, line width=1.2pt]  (24) to (25a);
                \draw [-, line width=1.2pt, in=0, out=30, looseness=1.75] (20) to (22.center);
		\draw [-, line width=1.2pt, bend left=315, looseness=1.25] (22.center) to (28.center);
		\draw [-, line width=1.2pt, in=630, out=-90, looseness=2.00] (18) to (11);
		\draw [-, line width=1.2pt]  (16) to (13);
             \draw [-, draw=black, line width=3pt, bend right=15] (29.center) to (30.center);
	     \draw [-, draw=black, line width=3pt, bend left=15] (31.center) to (32.center);
\draw [thick, rounded corners] (1.25, 0.75) rectangle (7.8, -3);
\end{tikzpicture}
        }
        \medskip
        
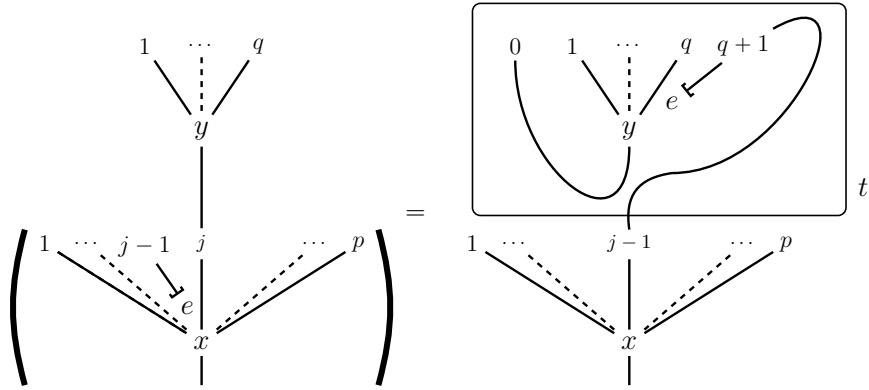
\captionof{figure}{The relation $(e \bullet_{j-1} x) \circ_j y = x \circ_{j-1} (e \bullet_0 y)$.}
        \end{center}
  To conclude, let us still depict the third line in relation \eqref{opere2}, which involves a possibly unexpected flip between $x$ and $y$:
     \begin{center}
       \scalebox{0.75}
                {
\begin{tikzpicture}
		\node [fill=white, draw=none, shape=rectangle] (0) at (0.25, 0) {\LARGE $=$};
		\node [fill=white, draw=none, shape=rectangle] (1) at (-2, 0) {\Large $p+1$};
		\node [fill=white, draw=none, shape=rectangle] (2) at (-2, 1.5) {\LARGE $y$};
		\node [fill=white, draw=none, shape=rectangle] (3) at (-3, 3) {\Large $1$};
		\node [fill=white, draw=none, shape=rectangle] (4) at (-2, 3) {$\cdots$};
		\node [fill=white, draw=none, shape=rectangle] (5) at (-1, 3) {\Large $q$};
		\node [fill=white, draw=none, shape=rectangle] (6) at (-3, 0) {\Large $p$};
		\node [fill=white, draw=none, shape=rectangle] (8) at (-5.25, 0) {};
		\node [fill=white, draw=none, shape=rectangle] (11) at (2, 0) {\Large $0$};
		\node [fill=white, draw=none, shape=rectangle] (12) at (4, 0) {$\cdots$};
		\node [fill=white, draw=none, shape=rectangle] (13) at (2, 1.5) {\LARGE $x$};
		\node [fill=white, draw=none, shape=rectangle] (14) at (1, 3) {\Large $1$};
		\node [fill=white, draw=none, shape=rectangle] (15) at (2, 3) {$\cdots$};
		\node [fill=white, draw=none, shape=rectangle] (16) at (3, 3) {\Large $p$};
		\node [fill=white, draw=none, shape=rectangle] (17) at (5, 0) {\Large $q$};
		\node [fill=white, draw=none, shape=rectangle] (18) at (4, -1.5) {\LARGE $y$};
		\node [fill=white, draw=none, shape=rectangle] (19) at (-3.25, -.9) {\LARGE $e$};
		\node [fill=white, draw=none, shape=rectangle] (20) at (6, 0) {\Large $q+1$};
		\node [fill=white, draw=none, shape=rectangle] (21) at (4.75, -1) {\LARGE $e$};
		\node [fill=white, draw=none, shape=rectangle] (22) at (4.75, -2.25) {};
		\node [fill=white, draw=none, shape=rectangle] (23) at (-4, -3.25) {};
		\node [fill=white, draw=none, shape=rectangle] (24) at (-4, -1.5) {\LARGE $x$};
		\node [fill=white, draw=none, shape=rectangle] (25) at (-5, 0) {\Large $1$};
		\node [fill=white, draw=none, shape=rectangle] (26) at (-4, 0) {$\cdots$};
		\node [fill=white, draw=none, shape=rectangle] (27) at (3, 0) {\Large $1$};
		\node [fill=white, draw=none, shape=rectangle] (28) at (4, -3.25) {};
                \node [] (t) at (8.1, -2.5) {\LARGE $t$};
                
                \node [] (29) at (-5.25, .5) {};
                \node [] (30) at (-5.25, -3.25) {};
                \node [] (31) at (-1, .5) {};
                \node [] (32) at (-1, -3.25) {};
                \node [] (33) at (1.5, .5) {};
                \node [] (34) at (1.5, -3.25) {};
                \node [] (35) at (7.75, .5) {};
                \node [] (36) at (7.75, -3.25) {};

		\draw [-, line width=1.2pt]  (3) to (2);
		\draw [dashed, line width=1.2pt]  (4) to (2);
		\draw [-, line width=1.2pt]  (5) to (2);
		\draw [-Bracket, line width=1.2pt]  (1) to (19);
		\draw [-, line width=1.2pt]  (6) to (24);
		\draw [dashed, line width=1.2pt]  (26) to (24);
		\draw [-, line width=1.2pt]  (25) to (24);
		\draw [-, line width=1.2pt]  (24) to (23.center);
		\draw [-, line width=1.2pt]  (2) to (1);
		\draw [-, line width=1.2pt]  (14) to (13);
		\draw [dashed, line width=1.2pt]  (15) to (13);
		\draw [-, line width=1.2pt]  (13) to (11);
		\draw [-, line width=1.2pt]  (27) to (18);
		\draw [dashed, line width=1.2pt]  (12) to (18);
		\draw [-, line width=1.2pt]  (17) to (18);
		\draw [-Bracket, line width=1.2pt]  (20) to (21);
		\draw [-, line width=1.2pt, in=0, out=30, looseness=1.75] (20) to (22.center);
		\draw [-, line width=1.2pt, bend left=315, looseness=1.25] (22.center) to (28.center);
		\draw [-, line width=1.2pt, in=630, out=-90, looseness=2.00] (18) to (11);
		\draw [-, line width=1.2pt]  (16) to (13);
             \draw [-, draw=black, line width=3pt, bend right=15] (29.center) to (30.center);
	     \draw [-, draw=black, line width=3pt, bend left=15] (31.center) to (32.center);
\draw [thick, rounded corners] (1.25, 0.75) rectangle (7.8, -3);
\end{tikzpicture}
        }
        \medskip
        
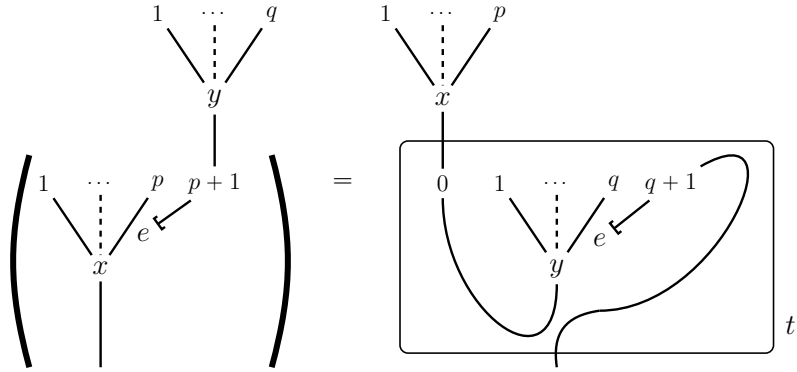
\captionof{figure}{The relation $(e \bullet_{p+1} x) \circ_{p+1} y = (e \bullet_0 y) \circ_1 x$.}
        \end{center}
At this point, the reader should not encounter any difficulties to picture all remaining instances in Eqs.~\eqref{opere1}--\eqref{opere6}. We can therefore proceed to our main objective in the subsequent section.

\section{BV algebras as cyclic duals of BV modules}

In this section, we are going to prove the core theorem, {\em i.e.}, that a cosimplicial-compatible operad $\cP$ in the sense of Definition \ref{compitompi} yields a Gerstenhaber structure on its cohomology, and if the operad is cocyclic-compatible as well, then this Gerstenhaber structure is, in particular, Batalin-Vilkoviski\u\i. We shall perform all proofs for the explicit cosimplicial operations \eqref{journuit} of the cyclic dual of a cyclic opposite $\cO$-module $\cM$, but it is clear from what was said in \S\ref{mendelssohn} that the same proofs carry over to the general case as well.

To start with, 
let us discuss the cyclic dual analogue of the graded commutativity of the cup product up to homotopy as described in \cite[Thm.~3]{Ger:TCSOAAR}.

\begin{lem}
  \label{kettensaege}
  Let $(\cM, \tau)$ be a cyclic opposite module over an operad $(\cO, \mu, e)$ with multiplication. If its cyclic dual $\hat\cM$ is at the same time a cosimplicial-com\-pat\-i\-ble operad with respect to the cosimplicial structure \eqref{journuit} induced by the opposite $\cO$-action, then for any $x \in \hat\cM(p)$ and  $y \in \hat\cM(q)$
  the cup product
  \begin{equation}
    \label{posaunehierposaunedort}
 x \cup y := (e \bullet_0 y) \circ_1 x 
        \end{equation}
  on $\hat \cM$ is an associative and graded commutative product up to homotopy, {\em i.e.},
  \begin{equation}
    \label{anachroniste}
y \cup x - (-1)^{\sss pq} x \cup y = (\gd x) \{ y \} + (-1)^{\sss q-1} x \{ \gd y \} - (-1)^{\sss q-1} \gd\big(x \{y \}\big), 
    \end{equation}
  where $\gd$ is the coboundary given in \eqref{stakker}.
  Finally, the graded Leibniz rule
  \begin{equation}
    \label{dga}
\gd(x \cup y) = \gd x \cup y + (-1)^p x \cup \gd y
    \end{equation}
holds, turning $(\hat\cM, \cup, \gd)$ into a dg algebra.
\end{lem}

\begin{rem}
 Observe that
by means of Eq.~\eqref{opere4}, the cup product  \eqref{posaunehierposaunedort}
can be written in two ways:
  \begin{equation}
    \label{posaunehierposauneda}
 x \cup y = (e \bullet_0 y) \circ_1 x = (e \bullet_{p+1} x) \circ_{p+1} y,
        \end{equation}
as is the case for the cup product in an operad with multiplication.
\end{rem}

\begin{proof}[Proof of Lemma \ref{kettensaege}]
The associativity of the product \eqref{posaunehierposaunedort} is easily seen by applying Eqs.~\eqref{opere1}, \eqref{opere2}, the composition axioms of an operad, as well as \eqref{SchlesischeStr}:
 \begin{eqnarray*}
(x \cup y) \cup z
& \overset{\sss \eqref{posaunehierposaunedort}}{=} &
(e \bullet_0 z) \circ_1 (x \cup y)
\ \, \overset{\sss \eqref{posaunehierposaunedort}}{=} \ \,
(e \bullet_0 z) \circ_1 \big( (e \bullet_0 y) \circ_1 x\big)
\\
   & \overset{\sss \eqref{danton}}{=} &
\big((e \bullet_0 z) \circ_1 (e \bullet_0 y) \big) \circ_1 x
\ \, \overset{\sss \eqref{opere5}}{=} \ \,
\pig(\big( e \bullet_1 (e \bullet_0 z)\big) \circ_2 y \pig) \circ_1 x
\\
   & \overset{\sss \eqref{SchlesischeStr}}{=} &
\pig(\big( e \bullet_0 (e \bullet_0 z)\big) \circ_2 y \pig) \circ_1 x
\ \ \, \overset{\sss \eqref{opere4}}{=} \ \
   \pig(e \bullet_0 \big((e \bullet_0 z) \circ_1 y\big)\pig) \circ_1 x
\\
   & \overset{\sss \eqref{posaunehierposaunedort}}{=} &
   \big(e \bullet_0 (y \cup z)\big) \circ_1 x
\ \ \, \overset{\sss \eqref{posaunehierposaunedort}}{=} \ \
x \cup ( y \cup z).
 \end{eqnarray*}
  The homotopy relation \eqref{anachroniste} is proven by using the relations \eqref{opere1}--\eqref{opere6}:
\begin{equation}
  \label{first}
  \begin{split}
    \gd\big(x \{y\}\big) &=
    \textstyle\sum\limits^{p}_{j=1} \, \sum\limits^{p+q}_{i=0} (-1)^{\sss (q-1)(j-1) + i} e \bullet_i ( x \circ_j y)
    \\
   & = 
    \textstyle\sum\limits^{p}_{j=1} \, \sum\limits^{j-1}_{i=0} (-1)^{\sss (q-1)(j-1) + i}  (e \bullet_i x) \circ_{j+1} y
    \\
    &
    \qquad
    +
\textstyle\sum\limits^{p}_{j=1} \, \sum\limits^{j+q-1}_{i=j} (-1)^{\sss (q-1)(j-1) + i} 
x \circ_j (e \bullet_{i-j+1} y)
        \\
        & \qquad 
        +
        \textstyle\sum\limits^{p}_{j=1} \, \sum\limits^{p+q}_{i=j+q} (-1)^{\sss (q-1)(j-1) + i} 
        (e \bullet_{i-q+1} x) \circ_{j} y
    \\
    & = 
    \textstyle\sum\limits^{p+1}_{j=2} \, \sum\limits^{j-2}_{i=0} (-1)^{\sss (q-1)j + i}
    (e \bullet_i x) \circ_{j} y
    +
        \textstyle\sum\limits^{p}_{j=1} \, \sum\limits^{q}_{i=1} (-1)^{\sss q(j-1) + i} x \circ_j (e \bullet_i y)
        \\
        & \qquad 
        +
            \textstyle\sum\limits^{p}_{j=1} \sum\limits^{p+1}_{i=j+1} (-1)^{\sss (q-1)j + i} (e \bullet_{i} x) \circ_{j} y,
  \end{split}
  \end{equation}
where in the third step we used re-indexing $j \mapsto j+1$ in the first, $i \mapsto i - j+1$ in the second, and $i \mapsto i -q +1$ in the third summand.
Observe that the second summand can be rewritten as
\begin{equation}
  \label{second}
  \begin{split}
&         \textstyle\sum\limits^{p}_{j=1} \, \sum\limits^{q}_{i=1} (-1)^{\sss q(j-1) + i} x \circ_j (e \bullet_i y)
    \\
    = \, &
   x \{ \gd y \} - \textstyle\sum\limits^p_{j=1} (-1)^{\sss q(j-1)} x \circ_j (e \bullet_0 y)
    - \textstyle\sum\limits^p_{j=1} (-1)^{\sss qj+1} x \circ_j (e \bullet_{q+1} y),
  \end{split}
  \end{equation}
whereas the first and the third summand in the last step of \eqref{first} are
\begin{equation}
  \label{third}
  \begin{split}
    &
   \textstyle\sum\limits^{p+1}_{j=2} \, \sum\limits^{j-2}_{i=0} (-1)^{\sss (q-1)j + i}
    (e \bullet_i x) \circ_{j} y
        +
            \textstyle\sum\limits^{p}_{j=1} \sum\limits^{p+1}_{i=j+1} (-1)^{\sss (q-1)j + i} (e \bullet_{i} x) \circ_{j} y,
\\
            = \ &
            (-1)^{\sss q-1} (\gd x) \{ y \}
            - \textstyle\sum\limits^{p+1}_{j=1} (-1)^{\sss q(j-1) + 1}
    (e \bullet_{j} x) \circ_j y
\\
& \quad
            - \textstyle\sum\limits^{p+1}_{j=1} (-1)^{\sss q(j-1)}
            (e \bullet_{j-1} x) \circ_j y
                    \\
            = \ &
            (-1)^{\sss q-1} (\gd x) \{ y \}
            - \textstyle\sum\limits^{p}_{j=1} (-1)^{\sss qj}
    x \circ_j  ( e \bullet_{q+1} y) - (-1)^{\sss q(p+1)} (e \bullet_0 y) \circ_{1} x
\\
& \quad
            - \textstyle\sum\limits^{p+1}_{j=2} (-1)^{\sss qj+1}
            x \circ_{j-1}  (e \bullet_{0} y) -  (-1)^{\sss q-1} (e \bullet_{q+1} y) \circ_{q+1} x,
              \\
            = \ &
            (-1)^{\sss q-1} (\gd x) \{ y \}
            + \textstyle\sum\limits^{p}_{j=1} (-1)^{\sss q(j-1) +1}
    x \circ_j  ( e \bullet_{q+1} y) + (-1)^{\sss px} x \cup y
\\
& \quad
            - \textstyle\sum\limits^{p}_{j=1} (-1)^{\sss q(j-1)+1}
  x \circ_{j}  (e \bullet_{0} y) - (-1)^{\sss q-1} y \cup x,
  \end{split}
  \end{equation}
where we used \eqref{opere4} in the third step and re-indexing $j \mapsto j-1$, as well as \eqref{posaunehierposaunedort} and \eqref{posaunehierposauneda} in the last. Observe that the second and the fourth summand in the last step of \eqref{third} are minus the second and third in \eqref{second}.
Hence, assembling the three equations \eqref{first}--\eqref{third}, we obtain
$$
\gd\big(x \{y \}\big) =
x \{ \gd y \} 
+ (-1)^{\sss q-1} (\gd x) \{ y \} - (-1)^{\sss q-1}
\big( y \cup x - (-1)^{\sss pq} x \cup y \big), 
$$
as desired.
Finally, to prove Eq.~\eqref{dga}, one again splits the respective sum according to the various cases in the relations \eqref{opere1} and uses, in particular, the second line in \eqref{opere2} to fill in the missing terms:
 \begin{eqnarray*}
\gd(x \cup y)
& \overset{\sss \eqref{stakker}, \eqref{posaunehierposaunedort}}{=} &
\Sum_{k=0}^{p+q+1} (-1)^{\sss k} e \bullet_k \big( (e \bullet_0 y) \circ_1 x) 
   \\
   & \overset{\sss \eqref{opere1}}{=} &
\big(e \bullet_0 (e \bullet_0 y)\big) \circ_2 x 
+
\Sum_{k=1}^{p} (-1)^{\sss k} (e \bullet_0 y) \circ_1 (e \bullet_k x)
\\
&& \qquad 
+
\Sum_{k=p+1}^{p+q+1} (-1)^{\sss k} \big(e \bullet_{k-p+1} (e \bullet_0 y)\big) \circ_1 x
   \\
   & \overset{\sss \eqref{SchlesischeStr} 
   }{=} &
\big(e \bullet_1 (e \bullet_0 y)\big) \circ_2 x 
+
\Sum_{k=1}^{p} (-1)^{\sss k} (e \bullet_0 y) \circ_1 (e \bullet_k x)
\\
&& \qquad 
+
\Sum_{k=p+1}^{p+q+1} (-1)^{\sss k} \big(e \bullet_0 (e \bullet_{k-p} y)\big) \circ_1 x
   \\
   & \overset{\sss \eqref{opere2}}{=} & 
\Sum_{k=0}^{p} (-1)^{\sss k} (e \bullet_0 y) \circ_1 (e \bullet_k x)
+
\Sum_{k=1}^{q+1} (-1)^{\sss p+k} \big(e \bullet_0 (e \bullet_{k} y)\big) \circ_1 x
   \\
   & \overset{\sss }{=} & 
\Sum_{k=0}^{p} (-1)^{\sss k} \gd_k x \cup y
+
(-1)^p \Sum_{k=1}^{q+1} (-1)^{\sss k} x \cup \gd_{k} y 
   \\
   & \overset{\sss \eqref{opere2}, \eqref{posaunehierposaunedort}}{=} & 
\Sum_{k=0}^{p} (-1)^{\sss k} \gd_k x \cup y + (-1)^{\sss p+1} (e \bullet_0 y) \circ_1 (e \bullet_{p+1} x)
\\
 && \qquad 
- (-1)^{\sss p+1} \big(e \bullet_1 (e \bullet_0 y)\big) \circ_1 x
+
(-1)^p \Sum_{k=1}^{q+1} (-1)^{\sss k} x \cup \gd_k y
   \\
   & \overset{\sss 
     \eqref{posaunehierposaunedort}}{=} &
   \gd x \cup y + (-1)^{\sss p} x \cup \gd y,
 \end{eqnarray*}
 which is the right hand side in \eqref{dga}, and
 where we used re-indexing $k \mapsto k-q$ in step four.
 \end{proof}

\begin{lemma}
  \label{leibnizuptohomotopy}
  With the same assumptions as in Lemma \ref{kettensaege}, one has
  \begin{equation}
    \label{cuphomo1}
(y \cup z)\{ x\} = y \{x\} \cup z + (-1)^{\sss (p-1)q} y \cup z\{x\}
  \end{equation}
  for any $x \in \hat\cM(p)$,  $y \in \hat\cM(q)$, and any $z \in \hat\cM$. Bracing the cup product from the left instead, one obtains
  \begin{equation}
\label{cuphomo2}
\begin{split}
x \{y \cup z\} &= 
(-1)^{\sss (p-1)r} x \{y\} \cup z + y \cup x\{z\}
\\
& \quad
+ (-1)^{\sss (q-1)r} \gd \big(F(x,y,z)\big) + (-1)^{\sss p+1} F(\gd x, y, z)
\\
& \quad
+ (-1)^{\sss p(q-1)} F(x, \gd y, z)
+ (-1)^{\sss (p+q)(r-1)} F(x, y, \gd z),
\end{split}
   \end{equation}
  where $F$ is defined by 
  $$
F(x,y,z) := \Sum^{p-1}_{i=1} \, \Sum^{p+q-1}_{j= q +i} (-1)^{\sss (q-1)(i-1)+(r-1)(j-1)} (x \circ_i y) \circ_j z 
  $$
  as an element in $\hat\cM(p+q+r-2)$.
  \end{lemma}

\begin{proof}
  The first identity \eqref{cuphomo1} is a simple and direct check
  using the defining Equation \eqref{posaunehierposaunedort} along with the relations \eqref{opere1} \& \eqref{opere2}. Checking \eqref{cuphomo2} is less direct, and much more tedious, but essentially still only relying on a double sum yoga exploiting the same identities from Definition \ref{compitompi} resp.\ \S\ref{mendelssohn}. Not being substantially different from the proof technique applied in Lemma \ref{kettensaege} resp.\ the forthcoming central Theorem \ref{main}, we omit this here.
  \end{proof}

The two preceding lemmata allow us to state:

\begin{prop}
  \label{gerstasgerstcan}
The cohomology groups of the cyclic dual $\hat \cM$ as in Lemma \ref{kettensaege} with respect to the differential \eqref{stakker} form a Gerstenhaber algebra.
  \end{prop}

\begin{proof}
  The product structure is given by the cup product \eqref{posaunehierposaunedort}, its graded commutativity on cohomology following from \eqref{anachroniste}, where all terms involving $\gd$ disappear; which, hence, results into 
$$
x \cup y = (-1)^{\sss pq} y \cup x.
$$
The (Gerstenhaber) bracket, as for any operad,
is defined by means of the graded commutator of the braces,
  $$
\{x, y \} := x\{y\} - (-1)^{\sss (p-1)(q-1)} y \{x\}
  $$
for $x \in \hat\cM(p)$, $y \in \hat\cM(q)$,  see \cite[Thm.~1]{Ger:TCSOAAR}, or Eq.~\eqref{bracket-po-packet} in \S\ref{pamukkale1}, which guarantees the (graded) antisymmetry and the (graded) Jacobi identity in a customary way, see {\em op.~cit.}\ for explicit details.
Finally, on cohomology all terms involving~$\gd$ in Eq.~\eqref{cuphomo2} as well disappear, and along with \eqref{cuphomo1} 
assemble to
\begin{equation*}
\{ x, y \cup z\} = \{x, y\} \cup z + (-1)^{\sss (p-1)q} y \cup \{x, z\}, 
  \end{equation*}
for any $z \in \hat\cM$, that is, the {Leibniz identity} \eqref{leibniz}. Hence $H^\bull(\hat \cM) := H^\bull(\hat\cM, \gd)$ turns into a Gerstenhaber algebra.
\end{proof}

Although all proofs  in this section so far have been performed, due to our principal objective, with respect to the cyclic dual $\hat\cM$  
by explicitly using the relations \eqref{opere1}--\eqref{unschoen2}, it is clear that the same proofs can be carried out, {\em mutatis mutandis}, by the analogous (and more general) relations \eqref{besser1}--\eqref{unschoen} in Definition \ref{compitompi}. Combining this observation with the fact that the associativity relations \eqref{SchlesischeStr} 
for an opposite module over an operad $(\cO, \mu, e)$ applied to the elements $\mu$ and $e$ are equivalent to the compatibility relations between cofaces and codegeneracies of the cosimplicial $k$-module given by Eqs.~\eqref{journuit}, 
we can immediately state: 

\begin{cor}
  \label{alsomain}
  Let $\cP$ be a cosimplicial-compatible nonsymmetric operad in the sense of Definition \ref{compitompi}. Then its cohomology groups $H(\cP)$, induced by the cosimplicial structure, form a Gerstenhaber algebra.
\end{cor}

\vspace*{.1cm}
\begin{center}
* \quad * \quad *
\end{center}
\vspace*{.15cm}

The next level of complexity will be to add the cyclic operator in order to promote the Gerstenhaber structures of Proposition \ref{gerstasgerstcan} and Corollary \ref{alsomain}, respectively, to those of BV algebras:

\begin{theorem}
  \label{main}
  Let $(\cO, \mu, e)$ be an operad with multiplication.
If the cyclic dual $\hat \cM$ of a cyclic opposite $\cO$-module $(\cM, t)$
is at the same time a cosimplicial and cocyclic-compatible operad in the sense of Definition \ref{compitompi},
then the identity
\begin{equation}
\label{nunjanaja2}
\begin{split}
  \{x,y\} &\, = \,
x\{y\} - (-1)^{\sss (p-1)(q-1)} y \{x\}
  \\
  & \, = \, 
(-1)^{\sss (q-1)p} B(y \cup x)
    - (-1)^{\sss p} Bx \cup y
        - (-1)^{\sss p(q-1)} By \cup x  
\\
& \, \, \quad
+ (-1)^{\sss pq} \delta(S_xy)  + (-1)^{\sss p} S_{\gd x} y - (-1)^{\sss pq} S_{x}\gd y
\\
  & \, \, \quad  - (-1)^{\sss q} \delta(S_y x)  + (-1)^{\sss p(q-1)} S_{\gd y} x  - (-1)^{\sss p+q} S_y \gd x
\end{split}
\end{equation}
%
holds
for all $x \in \hat\cM(p)$ and $y \in \hat\cM(q)$
on the normalised
complex, 
where
\begin{eqnarray}
  \nonumber
  \gd x &=& \Sum_{i=0}^{p+1} (-1)^i e \bullet_i x,
  \\[.2mm]
 \nonumber
  x \cup y & =& (e \bullet_0 y) \circ_{1} x, 
\\[.2mm]
  \nonumber
  B x & = & \Sum^{p-1}_{i=0} (-1)^{\sss \sss (p-1)(i-1)} t^i ( \mu \bullet_0 tx),
    \end{eqnarray}
as introduced before, along with
  \begin{eqnarray}
    \label{humanoid2}
  S_x y & := &
\Sum^{p-1}_{i=1} \Sum^i_{j=1} (-1)^{\scriptscriptstyle (q-1)i + (p+q)j + pq}  \, t^{j-1} (\mu \bullet_0 t x) \circ_{p-i+j-1} y
\end{eqnarray}
as an element in $\hat \cM(p+q-2)$. 
\end{theorem}

\begin{proof}
 The homotopy formula \eqref{nunjanaja2} is proven by a meticulous analysis of all appearing terms, which are subsequently shown to equal all out by means of an enhanced double sum yoga plus numerous summation re-indexings.
  We are going to prove this only for the case $p, q \geq 1$, the case of two zero cochains being trivial, whereas the case for either $p=0$ or $q=0$ is similar to what follows, but (much) simpler.
      
  Let us therefore develop all terms in Eq.~\eqref{nunjanaja2} by using the relations \eqref{opere1}--\eqref{unschoen2}.
  To start with, 
 \begin{eqnarray*}
   &&  (-1)^{\sss p(q-1)} B(y \cup x)
   \\
& \overset{\sss \eqref{posaunehierposaunedort}}{=} &
(-1)^{\sss p(q-1)}   B\big((e \bullet_0 x) \circ_1 y \big)
   \\
   & \overset{\sss \eqref{humanoid1}}{=} &
 \Sum^{p+q-1}_{i=0} (-1)^{\sss pq + (p-1)i +1} \, t^i \pig(\mu \bullet_0 t\big( (e \bullet_0 x) \circ_1 y)\big) \pig) 
  \\
   & \overset{\sss \eqref{unschoen2}, \eqref{lagrandebellezza1}}{=} &
\Sum^{p+q-1}_{i=0}  (-1)^{\sss pq + (p-1)i +1} \, t^i \pig(\mu \bullet_0 \big( (e \bullet_1 tx) \circ_2 y)\big) \pig) 
\\
   & \overset{\sss \eqref{SchlesischeStr}, \eqref{opere3}}{=} &
\Sum^{p+q-1}_{i=0}  (-1)^{\sss pq + (p-1)i +1} \, t^i (tx \circ_1 y) 
 \end{eqnarray*}
 \begin{eqnarray*}
& \overset{\sss \eqref{unschoen2} 
    }{=} &
    \Sum^{p-1}_{i=0}  (-1)^{\sss pq + (p-1)i +1} \, t^{i+1} x \circ_{i+1} y
    +
    \Sum^{p+q-1}_{i=p}  (-1)^{\sss pq + (p-1)i +1} \, t^{i} (tx \circ_{1} y) 
    \\
    & \overset{\sss \eqref{unschoen2} 
    }{=} &
    \Sum^{p}_{i=1}  (-1)^{\sss pq  + (p-1)(i-1)+1} \,  t^{i} x \circ_{i} y
    +
    \Sum^{q-1}_{i=0}  (-1)^{\sss pq  + (p-1)i+1} \,  t^{p+i} (tx \circ_{1} y) 
    \\
    & \overset{\sss \eqref{unschoen2}
    }{=} &
    \Sum^{p}_{i=1}  (-1)^{\sss  pq  + (p-1)i+p} \,  t^{i} x \circ_{i} y
    +
    \Sum^{q-1}_{i=0}  (-1)^{\sss  pq  + (p-1)i+1} \,  t^{i} (ty \circ_{1} t^{p+1} x) 
 \\
    & \overset{\sss \eqref{unschoen2}, \eqref{lagrandebellezza2}
    }{=} &
    \Sum^{p}_{i=1}  (-1)^{\sss  pq  + (p-1)i+p} \,  t^{i} x \circ_{i} y
    +
    \Sum^{q}_{i=1}  (-1)^{\sss  pq  + (p-1)i+p} \,  t^{i}y \circ_i  x. 
    \\
  &  =: & (1) + (2).
 \end{eqnarray*}
 Hence, up to signs, the relation \eqref{nunjanaja2} is symmetric in $x$ and $y$ and it shall turn  out that (with the respective signs) the brace $x \{y\}$ is generated by the terms (1) along with $Bx \cup y$, $\gd(S_xy)$, $S_{\gd x} y$, as well as $S_{x} \gd y$; whereas $y \{x\}$ is given by the terms with $x$ and $y$ exchanged. This implies that we only have to prove {\em half} of the homotopy formula, the other part being completely analogous.
 In fact,
 \begin{eqnarray*}
  - (-1)^{\sss p} Bx \cup y
& \overset{\sss \eqref{posaunehierposaunedort}, \eqref{humanoid1}}{=} &
\Sum^{p-1}_{i=0} (-1)^{\sss (p-1)i} \,  (e \bullet_0 y) \circ_1 t^i (\mu \bullet_0 tx) 
\\
& \overset{\sss \eqref{opere2} }{=} &
\Sum^{p-1}_{i=0} (-1)^{\sss  (p-1)i} \,  \big( e \bullet_p t^i (\mu \bullet_0 tx)\big) \circ_p y
 \end{eqnarray*}
 \begin{eqnarray*}
& \overset{\sss \eqref{lagrandebellezza1} }{=} &
\Sum^{p}_{i=1} (-1)^{\sss  (p-1)(i-1)} \,  t^{i-1} \big( e \bullet_{p-i+1} (\mu \bullet_0 tx)\big) \circ_p y
\\
& =: & (3),
\end{eqnarray*}
while, using a somewhat unconventional but nevertheless self-explanatory way of writing sums, one decomposes:
 \begin{eqnarray*}
   (-1)^{\sss pq} \gd(S_xy)
& \overset{\sss \eqref{humanoid2}, \eqref{stakker}}{=} &
   \Sum^{p-1}_{i=1} \Sum^{i}_{j=1} \Sum^{p+q-1}_{k=0}
   (-1)^{\sss q(i-j) + i + pj +k} \,  e \bullet_k \pig( t^{j-1} (\mu \bullet_0 tx) \circ_{p-i+j-1} y\pig) 
\\
& \overset{\sss  }{=} &
\Big(
\Sum^{p-1}_{i=1} \Sum^{i}_{j=1} \Sum^{p-i+j-2}_{k=0}
+ \Sum^{p-1}_{i=1} \Sum^{i}_{j=1} \, \Sum^{p+q-i+j-2}_{k=p-i+j-1}
+ \Sum^{p-1}_{i=1} \Sum^{i}_{j=1} \, \Sum^{p+q-1}_{k=p+q-i+j-1}
\Big)
\\
&& \hspace*{2.15cm}
(-1)^{\sss  q(i-j) + i + pj +k} \,  e \bullet_k \pig( t^{j-1} (\mu \bullet_0 tx) \circ_{p-i+j-1} y\pig) 
\\
& =: & (4) + (5) + (6).
 \end{eqnarray*}
 On these three terms let us apply \eqref{opere1} and subsequent re-indexing if need be, the outcome of which being:
\begin{eqnarray*}
   (4)
  & \overset{\sss \eqref{opere1}
   }{=} &
   \Sum^{p-1}_{i=1} \Sum^{i}_{j=1} \Sum^{p-i+j-2}_{k=0}
   (-1)^{\sss  q(i-j) + i + pj +k} \,  \big(e \bullet_k t^{j-1} (\mu \bullet_0 tx) \big) \circ_{p-i+j} y
\\
   & \overset{\sss 
   }{=:} & (4^\textrm{i}),
\end{eqnarray*} 
 as well as
 \begin{eqnarray*}
   (6)
   & \overset{\sss \eqref{opere1}
   }{=} &
   \Sum^{p-1}_{i=1} \Sum^{i}_{j=1} \, \Sum^{p+q-1}_{k=p+q-i+j-1}
   (-1)^{\sss  q(i-j) + i + pj +k} \,  \big( e \bullet_{k-q+1} t^{j-1} (\mu \bullet_0 tx) \big) \circ_{p-i+j-1} y
   \\
& \overset{\sss 
   }{=} &
   \Sum^{p-1}_{i=1} \Sum^{i}_{j=1} \, \Sum^{p}_{k=p-i+j}
   (-1)^{\sss  q(i-j) + i + pj +k+1} \,  \big( e \bullet_{k} t^{j-1} (\mu \bullet_0 tx) \big) \circ_{p-i+j-1} y
\\
   & \overset{\sss 
   }{=:} & (6^\textrm{i}),
   \end{eqnarray*} 
  along with
 \begin{eqnarray*}
   (5)
   & \overset{\sss \eqref{opere1}
   }{=} &
   \Sum^{p-1}_{i=1} \Sum^{i}_{j=1} \Sum^{q}_{k=1}
   (-1)^{\sss  q(i-j) + (p-1)j + p + k} \,  t^{j-1} (\mu \bullet_0 tx) \circ_{p-i+j-1} (e \bullet_k y)
   \\
    & \overset{\sss 
   }{=:} & (5^\textrm{i}).
\end{eqnarray*} 
 Next, note that $(5^\textrm{i})$ appears with opposite sign as a summand in the next term in the homotopy formula \eqref{nunjanaja2}:
 \begin{eqnarray*}
&&   - (-1)^{\sss pq} S_{x} \gd y
\\
   & \overset{\sss 
   \eqref{humanoid2}, \eqref{stakker}}{=} &
   \Sum^{p-1}_{i=1} \Sum^{i}_{j=1} \Sum^{q+1}_{k=0}
   (-1)^{\sss  q(i-j) + (p-1) j + p + k} \,  t^{j-1} (\mu \bullet_0 tx) \circ_{p-i+j-1} (e \bullet_k y)
   \\
    & \overset{\sss \eqref{opere2}}
   {=} &
     -(5^\textrm{i})
     +  \Sum^{p-1}_{i=1} \Sum^{i}_{j=1} 
     (-1)^{\sss q(i-j) + (p-1) j + p} \,  t^{j-1} (\mu \bullet_0 tx) \circ_{p-i+j-1} (e \bullet_0 y)
     \\
     &&
\qquad\ \
     +  \Sum^{p-1}_{i=1} \Sum^{i}_{j=1} 
   (-1)^{\sss q(i-j+1) + (p-1)(j-1)} \,  t^{j-1} (\mu \bullet_0 tx) \circ_{p-i+j-1} (e \bullet_{q+1} y).
  \end{eqnarray*} 
 Thus, the triple sum $(5^\textrm{i})$ disappears, whereas on the last two summands above apply \eqref{opere2} to obtain:
%
%
 \begin{eqnarray*}
&&
   \Sum^{p-1}_{i=1} \Sum^{i}_{j=1} 
   (-1)^{\sss q(i-j) + (p-1) j + p} \,  t^{j-1} \big(e \bullet_{p-i+j-1} t^{j-1} (\mu \bullet_0 tx)\big) \circ_{p-i+j} y   
     \\
     &&
\quad
+   \Sum^{p-1}_{i=1} \Sum^{i}_{j=1} 
     (-1)^{\sss q(i-j+1) + (p-1)(j-1)} \,  t^{j-1}  \big(e \bullet_{p-i+j-1} t^{j-1} (\mu \bullet_0 tx) \big) \circ_{p-i+j-1} y
\\
& \overset{\sss }
   {=}: &
(4^\textrm{ii}) + (7),
 \end{eqnarray*} 
 where one observes that one can reassemble $(4^\textrm{i})$ and $(4^\textrm{ii})$
 and afterwards decompose differently as
 \begin{eqnarray*}
   && (4^\textrm{i}) + (4^\textrm{ii})
\\
   & = &
   \Sum^{p-1}_{i=1} \Sum^{i}_{j=1} \Sum^{p-i+j-1}_{k=0}
   (-1)^{\sss q(i-j) + i+ pj + k} \,  \big(e \bullet_k t^{j-1} (\mu \bullet_0 tx) \big) \circ_{p-i+j} y
\\
&
\overset{\sss 
}{=}
&
\Big(
\Sum^{p-1}_{i=1} \Sum^{i}_{j=1} \Sum^{j-1}_{k=0}
+
\Sum^{p-1}_{i=1} \Sum^{i}_{j=1} \Sum^{p-i+j-1}_{k=j}
\Big)
   (-1)^{\sss q(i-j) + i+ pj + k} \,  \big(e \bullet_k t^{j-1} (\mu \bullet_0 tx) \big) \circ_{p-i+j} y
\\
&
\overset{\sss 
}{=:}
&
(4^\textrm{iii}) + (4^\textrm{iv}).
 \end{eqnarray*}
As a next step, let us isolate the summands $j=i$ and $j= i-1$ in $(4^\textrm{iii})$:
 \begin{eqnarray*}
(4^\textrm{iii})
& = &
   \Sum^{p-1}_{i=3} \Sum^{i-2}_{j=1} \Sum^{j-1}_{k=0}
   (-1)^{\sss q(i-j) + i+ pj + k} \,  \big(e \bullet_k t^{j-1} (\mu \bullet_0 tx) \big) \circ_{p-i+j} y
   \\
   &&
   \quad
   +
    \Sum^{p-1}_{i=1} \Sum^{i-1}_{k=0}
   (-1)^{\sss (p-1)i + k} \,  \big(e \bullet_k t^{i-1} (\mu \bullet_0 tx) \big) \circ_{p} y
    \\
    && \qquad
    +
     \Sum^{p-1}_{i=2} \Sum^{i-2}_{k=0}
   (-1)^{\sss q + p(i-1) + i + k} \,  \big(e \bullet_k t^{i-2} (\mu \bullet_0 tx) \big) \circ_{p-1} y
   \\
   &&
\quad 
   \\
&
\overset{\sss 
}{=:}
&
(4^\textrm{v}) + (4^\textrm{vi}) + (4^\textrm{vii}),
 \end{eqnarray*}
  and continue by isolating the summand $j=i$ in the triple sum $(4^\textrm{iv})$ as well:
 \begin{eqnarray*}
(4^\textrm{iv})
& = &
   \Sum^{p-1}_{i=2} \Sum^{i-1}_{j=1} \Sum^{p-i+j-1}_{k=j}
   (-1)^{\sss q(i-j) + i+ pj + k} \,  \big(e \bullet_k t^{j-1} (\mu \bullet_0 tx) \big) \circ_{p-i+j} y
   \\
   &&
   \quad
   +
    \Sum^{p-1}_{i=1} \Sum^{p-1}_{k=i}
   (-1)^{\sss (p-1)i +k} \,  \big(e \bullet_k t^{i-1} (\mu \bullet_0 tx) \big) \circ_{p} y
    \\
&
\overset{\sss 
}{=:}
&
(4^\textrm{viii}) + (4^\textrm{ix}).
 \end{eqnarray*}
 The last summand that appears in the homotopy formula \eqref{nunjanaja2} that has to be considered is $(-1)^{\sss p} S_{\gd x} y$, which, using $t ( e \bullet_{p+1} x) = e \bullet_0 x$ for $x \in \cM(p)$, that is, Eq.~\eqref{cine40}, can be expressed by means of Eqs.~\eqref{lagrandebellezza1} and \eqref{SchlesischeStr} as:
 \begin{eqnarray*}
&&   (-1)^{\sss p} S_{\gd x} y
\\
   & \overset{\sss 
   \eqref{humanoid2}, \eqref{stakker}}{=} &
   \Sum^{p}_{i=1} \Sum^{i}_{j=1} \Sum^{p+1}_{k=0}
   (-1)^{\sss q(i-j+1) + i+ (p-1)j + k + p(q-1)} \,  t^{j-1} \big(\mu \bullet_0 t (e \bullet_k x) \big) \circ_{p-i+j} y
   \\
   & \overset{\sss 
   }
   {=} &
   \Sum^{p}_{i=1} \Sum^{i}_{j=1} \Sum^{p}_{k=1}
   (-1)^{\sss q(i-j+1) + i+ (p-1)j + k + p(q-1)} \,  t^{j-1} \big(\mu \bullet_0 t (e \bullet_k x) \big) \circ_{p-i+j} y
   \\
   && \quad
   +
      \Sum^{p}_{i=1} \Sum^{i}_{j=1} 
   (-1)^{\sss q(i-j+1) + i+ (p-1)j + p(q-1)} \,  t^{j} x \circ_{p-i+j} y
      \\
      &&
      \quad\quad
      +
      \Sum^{p}_{i=1} \Sum^{i}_{j=1} 
   (-1)^{\sss q(i-j+1) + i+ (p-1)j + pq-1} \,  t^{j-1} x \circ_{p-i+j} y
   \\
   & \overset{\sss 
   }
   {=:} &
(8) + (9) + (10).
  \end{eqnarray*} 
 Closely checking, the last two summands partially cancel by re-indexing the sums first by $i \mapsto i -1$, $j \to j-1$, and then by $i \mapsto p-i+1$:
 \begin{eqnarray*}
   &&  (9) + (10)
   \\
   & \overset{\sss 
   }{=} &
       \Sum^{p-1}_{i=1} \Sum^{i}_{j=1} 
   (-1)^{\sss q(i-j+1) + i+ (p-1)j + p(q-1)} \,  t^{j} x \circ_{p-i+j} y
            +
         \Sum^{p}_{j=1} 
   (-1)^{\sss (p-1)(j-1) + pq} \,  t^{j} x \circ_{j} y
         \\
         &&  
         +
         \Sum^{p}_{i=2} \Sum^{i}_{j=2} 
   (-1)^{\sss q(i-j+1) + i+ (p-1)j + pq-1} \,  t^{j-1} x \circ_{p-i+j} y
              +
         \Sum^{p}_{i=1} 
   (-1)^{\sss (q-1)(p-i)} \,  x \circ_{p-i+1} y
          \\
   & \overset{\sss 
   }{=} &
         \Sum^{p}_{j=1} 
   (-1)^{\sss (p-1)(j-1) + pq} \,  t^{j} x \circ_{j} y
                +
         \Sum^{p}_{i=1} 
   (-1)^{\sss (q-1)(i-1)} x \circ_{i} y
          \\
          & \overset{\sss 
   }
            {=:} &
            (9^\textrm{i}) + (10^\textrm{i}),
  \end{eqnarray*} 
 and one observes that not only $(1) = - (9^\textrm{i})$, but in particular
 \begin{equation}
   \label{hierdoch}
(10^\textrm{i}) = x \{ y \}
 \end{equation}
 is the brace that constitutes `the first half' of the Gerstenhaber bracket $\{x,y\}$ in Eq.~\eqref{nunjanaja2}. We continue by using \eqref{SchlesischeStr} and \eqref{lagrandebellezza1} on the various summands of $(8)$ and subsequently isolating the summand $i=j$, as well as afterwards inside the latter isolating the summand for which $k=p-i+1$:
\begin{eqnarray*}
  (8) 
    & \overset{\sss 
      \eqref{SchlesischeStr}, \eqref{lagrandebellezza1}
    }{=} &
        \Sum^{p}_{i=1} \Sum^{i}_{j=1} \Sum^p_{k=1} 
       (-1)^{\sss q(i-j) + (p-1)j + i + k +1} \,  t^{j-1} \big( e \bullet_k (\mu \bullet_0 tx) \big)  \circ_{p-i+j} y
         \end{eqnarray*}
 \begin{eqnarray*}
        & \overset{\sss 
    }{=} &
        \Sum^{p}_{i=2} \Sum^{i-1}_{j=1} \Sum^p_{k=1} 
       (-1)^{\sss q(i-j) + (p-1)j + i + k +1} \,  t^{j-1} \big( e \bullet_k (\mu \bullet_0 tx) \big)  \circ_{p-i+j} y
        \\
        &&
        \quad
+        \Sum^{p-1}_{i=1} \Sum^{p-i}_{k=1} 
       (-1)^{\sss pi+k+1} \,  t^{i-1} \big( e \bullet_k (\mu \bullet_0 tx) \big)  \circ_{p} y
\\
&&
\quad\quad 
+        \Sum^{p}_{i=1}  
       (-1)^{\sss (p-1)(i-1) +1} \,  t^{i-1} \big( e \bullet_{p-i+1} (\mu \bullet_0 tx) \big)  \circ_{p} y
\\
&&
\quad\quad\quad
+        \Sum^{p}_{i=2} \, \Sum^{p}_{k=p-i+2} 
       (-1)^{\sss pi+k+1} \,  t^{i-1} \big( e \bullet_k (\mu \bullet_0 tx) \big)  \circ_{p} y
           \\
           & \overset{\sss 
    }
             {=:} &
             (8^\textrm{i}) + (8^\textrm{ii}) + (8^\textrm{iii}) + (8^\textrm{iv}).
   \end{eqnarray*} 
One notices that $(8^\textrm{iii}) = - (3)$, so this term cancels out with $-(-1)^{\sss p} \,  Bx \cup y$. On the other hand, using \eqref{lagrandebellezza1} again and re-indexing $k \mapsto k +i-1$, one computes
\begin{eqnarray*}
   (8^\textrm{ii})
    & \overset{\sss 
     \eqref{lagrandebellezza1}
    }{=} &
         \Sum^{p-1}_{i=1} \Sum^{p-i}_{k=1} 
         (-1)^{\sss pi+k+1} \,  \big( e \bullet_{k+i-1} t^{i-1} (\mu \bullet_0 tx) \big)  \circ_{p} y
         \\
          & \overset{\sss 
    }{=} &
         \Sum^{p-1}_{i=1} \Sum^{p-1}_{k=i} 
         (-1)^{\sss (p-1)i+k+1} \,  \big( e \bullet_{k} t^{i-1} (\mu \bullet_0 tx) \big)  \circ_{p} y
 \\
          & \overset{\sss 
    }{=} &
         -(4^\textrm{ix}),
     \end{eqnarray*} 
which therefore disappear.
In a similar way, but taking \eqref{cine40} into account, and re-indexing twice,  
\begin{eqnarray*}
   (8^\textrm{iv})
    & \overset{\sss 
     \eqref{lagrandebellezza1}, \eqref{cine40}
    }{=} &
         \Sum^{p}_{i=2} \Sum^{i-2}_{k=0} 
         (-1)^{\sss (p-1)i+k+1} \,  \big( e \bullet_{k} t^{i-2} (\mu \bullet_0 tx) \big)  \circ_{p} y
\\
          & \overset{\sss 
    }{=} &
         -(4^\textrm{vi}),
\end{eqnarray*} 
so these terms cancel out as well.
To elaborate on $(8^\textrm{i})$, we (unfortunately) have to split the sum again so as to (as we tacitly did for $(8^\textrm{iv})$ as well) take the {\em flip over} into account, {\em i.e.}, the fact that, for an element $z \in \cM(p-1)$, one has $t (e \bullet_k z) = e \bullet_{k+1} tz$ up to $k = p-1$ but $t (e \bullet_p z) = e \bullet_0 z$. Correspondingly, when applying the higher powers $t^{j-1}$ to $e \bullet_k z$, at the point $k > p - (j-1)$ one needs to restart from the extra composition map $e \bullet_0 z$ again. Taking this into consideration, and after multiple re-indexing $i \mapsto i -1$, $j \mapsto j-1$, and $k \mapsto k+j-1$, one computes
\begin{eqnarray*}
  (8^\textrm{i}) 
    & \overset{\sss 
      \eqref{lagrandebellezza1}
  }{=} &
  \Big(
  \Sum^{p}_{i=2} \Sum^{i-1}_{j=1} \Sum^{p-j+1}_{k=1} 
  +
  \Sum^{p}_{i=2} \Sum^{i-1}_{j=1} \, \Sum^{p}_{k=p-j+2} 
  \Big)
  \\
  && \qquad\qquad\qquad
(-1)^{\sss q(i-j) + (p-1)j + i + k +1} \,  t^{j-1} \big( e \bullet_k (\mu \bullet_0 tx) \big)
\circ_{p-i+j} y
 \\
 &
 \overset{\sss 
 }
        {=} &
 \Sum^{p}_{i=2} \Sum^{i-1}_{j=1} \, \Sum^{p-j+1}_{k=1} 
     (-1)^{\sss q(i-j) + (p-1)j + i + k +1} \,  \big( e \bullet_{k+j-1} t^{j-1} (\mu \bullet_0 tx) \big)  \circ_{p-i+j} y
\\
  &&
  \quad 
  +
  \Sum^{p}_{i=3} \Sum^{i-1}_{j=2} \, \Sum^{j-2}_{k=0} 
    (-1)^{\sss q(i-j) + p(j-1) + i + k} \,  \big( e \bullet_{k} t^{j-2} (\mu \bullet_0 tx) \big)  \circ_{p-i+j} y
\\
&
\overset{\sss 
    }
  {=} &
 \Sum^{p-1}_{i=1} \Sum^{i}_{j=1}  \, \Sum^{p}_{k=j} 
     (-1)^{\sss q(i-j) + pj + i + k +1} \,  \big( e \bullet_{k} t^{j-1} (\mu \bullet_0 tx) \big)  \circ_{p-i+j-1} y
\\
 &&
  \quad 
  +
  \Sum^{p}_{i=3} \Sum^{i-1}_{j=2} \, \Sum^{j-2}_{k=0} 
    (-1)^{\sss q(i-j) + p(j-1) + i + k} \,  \big( e \bullet_{k} t^{j-2} (\mu \bullet_0 tx) \big)  \circ_{p-i+j} y
\\
           & \overset{\sss 
    }
             {=:} &
             (8^\textrm{v}) + (8^\textrm{vi}).
   \end{eqnarray*} 
One then has, by splitting up the summand for which $j = i-1$ from the second term above,
\begin{eqnarray*}
  (8^\textrm{vi}) 
    & \overset{\sss 
  }{=} &
  \Sum^{p}_{i=4} \Sum^{i-2}_{j=2} \, \Sum^{j-2}_{k=0} 
    (-1)^{\sss q(i-j) + pj + i + k} \,  \big( e \bullet_{k} t^{j-2} (\mu \bullet_0 tx) \big)  \circ_{p-i+j} y
\\
&&
\quad
+  \Sum^{p}_{i=3} \, \Sum^{i-3}_{k=0} 
    (-1)^{\sss q + (p-1)i + k} \,  \big( e \bullet_{k} t^{i-3} (\mu \bullet_0 tx) \big)  \circ_{p-1} y,
\end{eqnarray*}
which, by re-indexing $i \mapsto i-1$, $j \mapsto j-1$ in the first summand resp.\ $i \mapsto i-1$ in the second, is seen to cancel out with
$(4^\textrm{v})$ resp.\ $(4^\textrm{vii})$.

At this point, we are left to deal with the terms $(2), (4^\textrm{viii}), (6^\textrm{i}), (7)$, and $(8^\textrm{v})$. For the last four one has:
\begin{eqnarray*}
 (8^\textrm{v}) + (6^\textrm{i}) 
     \!\!\!& \overset{\sss 
  }{=} \!\!\!&
  \Sum^{p-1}_{i=1} \Sum^{i}_{j=1} \, \Sum^{p-i+j-1}_{k=j} 
    (-1)^{\sss q(i-j) + pj + i + k +1} \,  \big( e \bullet_{k} t^{j-1} (\mu \bullet_0 tx) \big)  \circ_{p-i+j-1} y
\\
 \!\!\! & \overset{\sss 
  }{=} \!\!\! &
  \Sum^{p-2}_{i=1} \Sum^{i}_{j=1} \, \Sum^{p-i+j-2}_{k=j} 
    (-1)^{\sss q(i-j) + pj + i + k +1} \,  \big( e \bullet_{k} t^{j-1} (\mu \bullet_0 tx) \big)  \circ_{p-i+j-1} y
\\
  &&
\!\!  +
  \Sum^{p-1}_{i=1} \Sum^{i}_{j=1}  
  (-1)^{\sss q(i-j+1) + (p-1)(j-1) +1} \,  \big( e \bullet_{p-i+j-1} t^{j-1} (\mu \bullet_0 tx) \big)  \circ_{p-i+j-1} y
  \\
\!\!\!    & \overset{\sss 
  }{=} \!\!\! &
  -(4^\textrm{viii}) - (7), 
\end{eqnarray*}
where the equality with respect to the first summand is obtained by re-indexing $i \mapsto i-1$ again. The only remaining term is now $(2)$, which, in a (up to signs) totally symmetric way  cancels by 
an analogous and equally long computation regarding the terms where $y$ precedes $x$ in the homotopy formula \eqref{nunjanaja2}, which we, understandably, omit; but which, exactly as in \eqref{hierdoch}, produces the `second half' of the Gerstenhaber bracket, {\em i.e.}, the brace $- (-1)^{\sss (p-1)(q-1)} y \{x\}$. This proves the relation \eqref{nunjanaja2} and therefore concludes the proof.
\end{proof}

As an immediate consequence one has:

\begin{cor}
  \label{barbarano}
Under the assumptions of Theorem \ref{main}, the cyclic dual of a cyclic opposite module
  $(\cM, \tau)$ induces a cochain complex $(\hat\cM^\bull, \gd)$ the cohomology of which is a BV algebra.
\end{cor}

\begin{proof}
  This is, in a standard way, a direct consequence of the homotopy formula \eqref{nunjanaja2} just proven combined with the homotopy relation \eqref{anachroniste}: on cohomology, all terms that involve the differential $\gd$ disappear, and Eq.~\eqref{nunjanaja2} reduces to
  \begin{equation}
    \label{wirklichsimpel1}
    \{x,y\}
    =
(-1)^{\sss (q-1)p} B(y \cup x)
    - (-1)^{\sss p} Bx \cup y
        - (-1)^{\sss p(q-1)} By \cup x,  
  \end{equation}
  whereas from \eqref{anachroniste} one derives $x \cup y = (-1)^{\sss pq} y \cup x$ and $B y \cup x =(-1)^{\sss p(q-1)} x \cup B y$
  on cohomology.
    With this, Eq.~\eqref{wirklichsimpel1} turns into
    \begin{equation*}
  \{x,y\}
    =
(-1)^{\sss p} B(x \cup y)
    - (-1)^{\sss p} Bx \cup y
        - x \cup By,  
      \end{equation*}
    which is the customary BV relation as quoted in Eq.~\eqref{bvbv}.
\end{proof}

Generalising this result in the spirit of Corollary \ref{alsomain}, we conclude:

\begin{cor}
  \label{alsoalsomain}
  Let $\cP$ be a cosimplicial and cocyclic-compatible nonsymmetric operad in the sense of Definition \ref{compitompi}. Then the cohomology groups $H(\cP)$, obtained from the cosimplicial structure, form a BV algebra.
  \end{cor}

\begin{proof}
  Analogously to what was said right above Corollary \ref{alsomain}, the proof of Theorem \ref{main} does not depend on the precise form of the cofaces and codegeneracies, but on three things only:
  first, again on the fact that the associativity relations \eqref{SchlesischeStr} 
  are equivalent to the compatibility relations between cofaces and codegeneracies in the cosimplicial $k$-module given by Eqs.~\eqref{journuit}; second,  on the relations \eqref{opere1}--\eqref{opere6}, which are simply the relations \eqref{besser1}--\eqref{besser3} for this specific situation; plus, third, the cyclic compatibility \eqref{lagrandebellezza1} for a cyclic opposite module over an operad with multiplication $(\cO, \mu, e)$ with respect to the two special elements $\mu$ and $e$ only. Let us show that these amount to the identities required for a cocyclic module to hold (see, for example, \cite[\S6]{Lod:CH}) when considering the cofaces and codegeneracies as in \eqref{journuit}, along with the cocyclic operator $\tau = t^{-1}$. Rewriting, by means of $t^{n+1} = \id$ in degree $n$, the relation \eqref{lagrandebellezza1} for $\tau$ and $\phi = \mu$ or $\phi = e$, one obtains
 \begin{equation*}
  \begin{array}{rcll}
\tau(\mu \bullet_i x) &=& \mu \bullet_{i-1} \tau x & \mbox{for} \ 1 \leq i \leq n-1, 
\\
\tau(e \bullet_i x) &=& e \bullet_{i-1} \tau x & \mbox{for} \ 1 \leq i \leq n+1, 
\end{array}
 \end{equation*}
 and with \eqref{journuit} it follows that
 $\tau\gs_i = \gs_{i-1} \tau$ for  $1 \leq i \leq n-1$ 
 as well as
 $\tau\gd_i = \gd_{i-1} \tau$ for  $1 \leq i \leq n+1$, which are the compatibility conditions in a cocyclic module for $i \neq 0$. We are left with checking
 $$
 \tau \gs_0 (x) = t^{-1} (\mu \bullet_0 x) = t^{n-1} (\mu \bullet_0 x)
 = \mu \bullet_{n-1} t^{n-1} x = \mu \bullet_{n-1} t^{-2} x
 = \gs_{n-1} \tau^2 (x) 
 $$
for $x \in \cM(n)$,  and likewise one computes
 $$
\tau \gd_0(x) = t^{-1}(e \bullet_0 x) = t^{n+1}(e \bullet_0 x) = e \bullet_{n+1} x
= \gd_{n+1} (x).
$$
Hence, what we in addition used in the proof of Theorem \ref{main} were the respective compatibility relations for a cocyclic $k$-module in disguise, so to speak.
  This makes it clear that one can rewrite the entire proof of Theorem \ref{main} in a more general setting for {\em any}
  operad $\cP$ in $\kmod$ that is also a cocyclic $k$-module respecting the cosimplicial and cocyclic compatibilities of Definition \ref{compitompi}.
\end{proof}
  
\section{Examples}
\label{examples}
Concrete examples where one can observe the statements proven above that motivated to develop this formal approach were given in \cite{Kow:ANCCOTCDOE} in the context of Hopf algebroids. We shall give a short summary in the more restricted context of Hopf algebras since properly introducing the many technical details involved in Hopf algebroid theory would go beyond the scope of the general level of this note. However, observe that only the latter allow to include the case of Hochschild cohomology for an associative algebra $A$ when looking for Gerstenhaber or BV algebra structures since the respective cohomology groups $\Ext_\Ae(A,A)$ can only be seen as a cohomology theory over a certain Hopf algebroid, but not over a Hopf algebra.

Standing assumption in this section, let $(H, k, \mu, \eta, \gD, \gve, S)$ be a Hopf algebra over $k$ (with characteristic zero if need be), where
$\mu$ and $\gD$ denote its product resp.\ coproduct,
$\eta$ and $\gve$ its unit resp.\ counit, and finally $S$ an involutive antipode, that is, $S^2 = \id$. The customary Sweedler notation for the coproduct will be used throughout and
elements in tensor powers will be (most of the time) written as $n$-tuples separated by commata instead of denoting them by notationally clumsier tensor chains. Ultimately, recall that the category $\hmod$ of left $H$-modules is monoidal a fact that is  reflected by the diagonal left $H$-action, denoted
\begin{equation}
  \label{mancino}
\begin{array}{rcl}
  h \mancino (m \otimes m') &:=&
  h_{(1)} m \otimes h_{(2)} m', \qquad m \in M, \, m' \in M', \, h \in H,
\end{array}
\end{equation}
for two objects $M, M' \in \hmod$ with action written by juxtaposition.

\subsection{$\bfCotor$ acting on $\bfCoext$ with $\bfExt$ as a resulting BV algebra}
Let us briefly apply the results in \cite[\S6]{Kow:ANCCOTCDOE} to see how the chain complex computing $\Coext$-groups can be seen as a cyclic opposite module over the cochain complex computing $\Cotor$-groups considered as an operad with multiplication, and then understand how the cyclic dual turns into a BV algebra in the spirit of the preceding sections. For details on the mentioned cohomology groups as derived functors of the {\em cohomomorphism functor} with respect to {\em contramodules} resp.\ the {\em cotensor functor} with respect to comodules, see {\em op.~cit.}, \S3.3 resp.\ \S3.2, in the context needed here, or \cite{Pos:HAOSAS} for the general theory.
Set
\begin{equation}
  \label{jupp}
\cO(p) := H^{\otimes p}, \qquad \cM(n) := \Hom(H^{\otimes n}, k),
\end{equation}
where in degree zero we put $\cO(0) := \cM(0) := k$.
The operadic structure on $\cO$ has been described, possibly first, in \cite[p.~65]{GerSch:ABQGAAD}; we do not need the explicit (partial) compositions here but only the fact that it is an operad with multiplication with the three distinguished elements $(\mu, \mathbb{1}, e) = \big( 1 \otimes 1, 1_H, 1_k)$. Specialising Lemma 6.1 \& Theorem 6.2 in \cite{Kow:ANCCOTCDOE} to the Hopf algebra case, we extract from there:
\begin{prop}
For all $1 \leq i \leq n-p+1$ and $0 \leq p \leq n$,
the operations
\begin{equation*}
    \bullet_i \colon \cO(p) \otimes \cM(n) \to \cM(n-p+1), 
\end{equation*}
given by, for any $(h^1, \ldots, h^p) \in \cO(p)$ and $\phi \in \cM(n)$,
\begin{equation*}
  \begin{split}
  & \big((h^1, \ldots, h^p) \bullet_i \phi\big)(g^1, \ldots, g^{n-p+1})
    := \phi\big(g^1, \ldots, g^i \mancino (h^1, \ldots, h^p), \ldots, g^{n-p+1}\big),
  \end{split}
\end{equation*}
where $g^j \in H$,
and declared to be zero if $p > n$, along with
\begin{equation*}
\big(1_k \bullet_i \phi\big)(g^1, \ldots, g^{n+1}) := \phi\big(g^1 , \ldots,  \gve(g^i), \ldots , g^{n+1}\big)
\end{equation*}
for elements in $\cO(0) \!=\! k$ and $1 \! \leq \! i \! \leq \! n+1$, induce on $\cM$ the structure of a uni\-tal opposite $\cO$-module.  
Moreover, defining an extra operation
\begin{equation*}
  \begin{split}
    \big((h^1, \ldots, h^p) \bullet_0 \phi\big)(g^1, \ldots, g^{n-p+1}) 
    &=
\phi\big( Sh^1 \mancino (h^2, \ldots, h^p, g^1,  \ldots, g^{n-p+1}) \big),
  \end{split}
  \end{equation*}
declared to be zero if $p > n + 1$, along with the cyclic operator 
\begin{equation}
  \label{cyconcoext}
  \begin{split}
    (t \phi)(g^1, \ldots, g^n) &=
\phi\big(Sg^1 \mancino (g^2, \ldots, g^n, 1)\big), 
  \end{split}
  \end{equation}
turns the opposite $\cO$-module $\cM$ into a cyclic one in the sense of Definition \ref{molck}.
\end{prop}

In particular, $\cM$ becomes a $k$-simplicial module, and a short computation reveals the faces and degeneracies. Putting as a shorthand $x := (g^1, \ldots, g^{n-1})$ and $y := (g^1, \ldots, g^{n+1})$ for $g^j \in H$, one obtains: 
\begin{equation}
  \label{simponcoext}
\begin{array}{rcccll}
(d_0\phi)(x) &=&      (\mu \bullet_0 \phi)(x)
  &=& \phi(1, g^1 , \ldots , g^{n-1}), &
  \\[1mm]
(d_i\phi)(x) &=&
(\mu \bullet_i \phi)(x)
&=& \phi(g^1 , \ldots , \gD g^i , \ldots , g^{n-1}), &
    \\[1mm]
(d_{n}\phi)(x) &=&    (\mu \bullet_0 t \phi)(x)
    &=& \phi(g^1 , \ldots , g^{n-1} , 1), &
\\[1mm]
(s_j\phi)(y) &=&
(e \bullet_{j+1} \phi)(y)
&=&
\phi\big(g^1 ,  \ldots ,  g^j \gve(g^{j+1}) ,  \ldots ,  g^{n+1}\big), &
\end{array}
\end{equation}
for $1 \leq i \leq n-1$ and $0 \leq j \leq n$, where $\phi \in \cM(n)$. This allows to compute the extra degeneracy by means of
\begin{equation}
  \label{kaltkaltkalt}
(s_{-1}\phi)(y) = (e \bullet_{0} \phi)(y)
  =
  t \big(e \bullet_{n+1} \phi\big)(y)
  =
\phi\big(Sg^1 \mancino (g^2, \ldots, g^{n+1})\big).
  \end{equation}
According to Definition \ref{moulesfrites}, one can then form the cyclic dual $\hat\cM$ as a cochain complex induced by the following cofaces and codegeneracies: 
\begin{equation}
  \label{cosimponcoext}
\begin{array}{rcll}
  (\gd_0\phi)(y) &=& \phi\big(Sg^1 \mancino (g^2, \ldots, g^{n+1})\big), &
  \\[1mm]
  (\gd_i\phi)(y) &=&
  \phi\big(g^1 ,  \ldots ,  g^{i-1} \gve(g^{i}) ,  \ldots ,  g^{n+1}\big),
  & \qquad 1 \leq i \leq n +1,
   \\[1mm]
(\gs_0\phi)(x) &=& \phi(1, g^1 , \ldots , g^{n-1}), &
\\[1mm]
(\gs_j\phi)(x) &=& \phi(g^1 , \ldots , \gD g^j , \ldots , g^{n-1}),
& \qquad 1 \leq j \leq n-1, 
\end{array}
\end{equation}
which, in addition, yields a cocyclic $k$-module when adding the inverse of \eqref{cyconcoext}, 
\begin{equation}
\label{inverset}
(t^{-1} \phi)(g^1, \ldots, g^n) = \phi \big(Sg^n \mancino (1, g^1, \ldots, g^{n-1}) \big) 
\end{equation}
in degree $n$. As a consequence, one can prove:

\begin{lem}
\label{ilfautdesjours}
  The cyclic dual $\hat\cM$ 
  of the sequence $\cM = \{ \cM(n)\}_{n \geq 0}$ of $k$-modules, where $\cM(n) = \Hom(H^{\otimes n}, k)$, with respect to the cyclic $k$-module defined by the relations \eqref{cyconcoext} and \eqref{simponcoext}, yields a cochain complex computing $\Ext^\bull_H(k,k)$. Defining on $\hat\cM$ the partial operadic composition
  \begin{small}
  \begin{equation}
\begin{split}
    \label{verticalcoext}
    & (\phi \circ_j \psi)(z)
    \\
    &
    :=
    \begin{cases}
      \phi\pig(\psi(g^1, \ldots, g^{q-1}, g^q_{(1)}) g^q_{(2)}, g^{q+1}, \ldots, g^{p+q-1}\pig)
      & \mbox{if} \ j = 1,
\\
\begin{array}{r}
  \hspace*{-.2cm}
  \phi\pig(g^{1}, \ldots, g^{j-2},
  g^{j-1}_{(1)}
  \psi\big(Sg^{j-1}_{(2)} \mancino (g^{j}, \ldots, g^{j+q-2}, g^{j+q-1}_{(1)})\big), 
\\
  g^{j+q-1}_{(2)},
  g^{j+q}, \ldots, g^{p+q-1}\pig) 
\end{array}
& \mbox{if} \  2 \leq j \leq p,
    \end{cases}
    \end{split}
  \end{equation}
  \end{small}for $\phi \in \hat\cM(p), \psi \in \hat\cM(q)$, and
  $z := (g^1, \ldots, g^{p+q-1})$,
  one obtains on $\hat\cM$ the structure of a cosimplicial and cocyclic-compatible operad in $\kmod$.
  \end{lem}

\begin{rem}
If $\psi$ is a $0$-cochain, with $\Hom(H^{\otimes \hskip 1pt 0}, k) \simeq k$ one puts $\psi = 1_k$, and Eq.~\eqref{verticalcoext} has to be read as
$
  (\phi \circ_1 1_k)(g^1, \ldots, g^{p-1}) :=
  \phi(1, g^1, \ldots, g^{p-1})
  $
  along with
  $
  (\phi \circ_j 1_k)(g^1, \ldots, g^{p-1}) :=
  \phi(g^1, \ldots, \gD g^{j-1}, \ldots, g^{p-1})
  $
  for all $2 \leq j \leq p$.
  \end{rem}

\begin{proof}[Proof of Lemma \ref{ilfautdesjours}]
That the cochain complex induced by Eqs.~\eqref{cosimponcoext} computes $\Ext^\bull_H(k,k)$ directly follows by specialising \cite[\S5]{Kow:ANCCOTCDOE} to Hopf algebras and trivial coefficients combining them with a cochain isomorphism obtained from a higher order Hopf-Galois (resp.\ translation) map, the details of which being omitted here. 

That the partial compositions \eqref{verticalcoext} yield an operadic structure in the sense of \S\ref{pamukkale1}, or, more concretely, fulfil the relations \eqref{danton} is a somewhat longish but standard check using the properties of an antipode along with $S = S^{-1}$ in this case, similar to the computations that follow below. 

Next, we prove that the vertical composition \eqref{verticalcoext} is cosimplicial-compatible with respect to the cosimplicial operators in Eqs.~\eqref{cosimponcoext} in the sense of Definition \ref{compitompi}. This is equally straightforward; we will therefore only check a few of them in order to illustrate the computations involved. For example, for $i = 0$ and $2 \leq j \leq p$, one has
\begin{equation*}
  \begin{split}
& \big(\gd_0 (\phi \circ_j \psi)\big)(g^1, \ldots, g^{p+q}) 
    \\
    = \  &
    (\phi \circ_j \psi)\big(Sg^1 \mancino (g^2, \ldots, g^{p+q})\big)
    \\
    = \  &
    (\phi \circ_j \psi)\pig(Sg^1_{(3)} \mancino (g^2, \ldots, g^{j}), 
    Sg^1_{(2)} \mancino (g^{j+1}, \ldots, g^{j+q}),
Sg^1_{(1)} \mancino (g^{j+q+1}, \ldots, g^{p+q})
\pig)
    \\
    = \  &
    \phi\pig(Sg^1_{(4)} \mancino (g^2, \ldots, g^{j-1}, g^{j}_{(1)})
    \psi\big(
    S( Sg^1_{(3)}Sg^{j}_{(2)}) \mancino Sg^1_{(2)} \mancino (g^{j+1}, \ldots, g^{j+q-1},  g^{j+q}_{(1)})\big),
    \\
    &
    \qquad \qquad \qquad \qquad \qquad \qquad  \qquad \qquad \qquad \qquad
Sg^1_{(1)} \mancino (g^{j+q}_{(2)}, g^{j+q+1}, \ldots, g^{p+q})\pig)
\end{split}
 \end{equation*}
 \begin{equation*}
\begin{split}
    = \  &
    \phi\pig(Sg^1_{(2)} \mancino (g^2, \ldots, g^{j-1}, g^{j}_{(1)})
    \psi\big(
    Sg^{j}_{(2)} \mancino (g^{j+1}, \ldots, g^{j+q-1},  g^{j+q}_{(1)})\big),
    \\
    &
    \qquad \qquad \qquad \qquad \qquad \qquad  \qquad \qquad \qquad \qquad
Sg^1_{(1)} \mancino (g^{j+q}_{(2)}, g^{j+q+1}, \ldots, g^{p+q})\pig)
   \\
    = \  &
    \gd_0 \phi\pig(g^1, \ldots, g^{j-1}, g^{j}_{(1)}
    \psi\big(
    Sg^{j}_{(2)} \mancino (g^{j+1}, \ldots, g^{j+q-1},  g^{j+q}_{(1)})\big),
g^{j+q}_{(2)}, g^{j+q+1}, \ldots, g^{p+q}\pig)
  \\
    = \  &
    \big((\gd_0 \phi) \circ_{j+1} \psi\big)\big(g^1, \ldots, g^{p+q}\big),
       \end{split}
\end{equation*}
and similarly (but easier) for $j =1$, as well as for the cases $1 \leq i \leq j-1$. This proves the first line in \eqref{besser1}, the other two being left to the reader, along with all relations in \eqref{besser2}. Let us, however, still check  the middle line in \eqref{besser3}: to begin with, let $j - 1 < i < j + q -2$ and $2 \leq j \leq p$. We have:
\begin{small}
\begin{equation*}
  \begin{split}
& \big(\gs_i (\phi \circ_j \psi)\big)(g^1, \ldots, g^{p+q-2}) 
    \\
    = \  &
    (\phi \circ_j \psi)\big(g^1, \ldots, \gD g^i, \ldots, g^{p+q})\big)
    \\
    = \  &
  \phi\pig(g^{1}, \ldots, g^{j-2},
  g^{j-1}_{(1)}
  \psi\big(Sg^{j-1}_{(2)} \mancino (g^{j}, \ldots, \gD g^i, \ldots, g^{j+q-2}, g^{j+q-1}_{(1)})\big),
g^{j+q-1}_{(2)},
  g^{j+q}, \ldots, g^{p+q-2}\pig) 
   \\
    = \  &
  \phi\pig(g^{1}, \ldots, g^{j-2},
  g^{j-1}_{(1)}
  \big(\gs_{i-j+1} \psi\big)\big(Sg^{j-1}_{(2)} \mancino (g^{j}, \ldots, g^{j+q-2}, g^{j+q-1}_{(1)})\big),
g^{j+q-1}_{(2)},
  g^{j+q}, \ldots, g^{p+q-2}\pig) 
   \\
    = \  &
  \big(\phi \circ_j (\gs_{i-j+1} \psi)\big)\big(g^1, \ldots, g^{p+q-2}\big).
  \end{split}
\end{equation*}
\end{small}The case of $i = j+q-1$ follows easily. As for the case $i = j-1$,
\begin{small}
\begin{equation*}
  \begin{split}
& \big(\gs_{j-1} (\phi \circ_j \psi)\big)(g^1, \ldots, g^{p+q-2}) 
    \\
    = \  &
    (\phi \circ_j \psi)\big(g^1, \ldots, \gD g^{j-1}, \ldots, g^{p+q})\big)
    \\
    = \  &
  \phi\pig(g^{1}, \ldots, g^{j-2},
  g^{j-1}_{(1)}
  \psi\big(Sg^{j-1}_{(2)} \mancino (g^{j-1}_{(3)}, g^j, \ldots, g^{j+q-2}, g^{j+q-1}_{(1)})\big),
g^{j+q-1}_{(2)},
  g^{j+q}, \ldots, g^{p+q-2}\pig) 
\\
    = \  &
  \phi\pig(g^{1}, \ldots, g^{j-2},
  g^{j-1}_{(1)}
  \psi\big(1, Sg^{j-1}_{(2)} \mancino (g^j, \ldots, g^{j+q-2}, g^{j+q-1}_{(1)})\big),
g^{j+q-1}_{(2)},
  g^{j+q}, \ldots, g^{p+q-2}\pig) 
   \\
   = \  &
  \phi\pig(g^{1}, \ldots, g^{j-2},
  g^{j-1}_{(1)}
  \big(\gs_0 \psi\big)\big(Sg^{j-1}_{(2)} \mancino (g^{j}, \ldots, g^{j+q-2}, g^{j+q-1}_{(1)})\big),
g^{j+q-1}_{(2)},
  g^{j+q}, \ldots, g^{p+q-2}\pig) 
   \\
    = \  &
  \big(\phi \circ_j (\gs_0 \psi)\big)\big(g^1, \ldots, g^{p+q-2}\big).
    \end{split}
\end{equation*}
\end{small}Similar computations can be made for $j =1$, which proves the middle line in \eqref{besser3}, the explicit check of the other two again being omitted.
Let us conclude by verifying \eqref{unschoen}, that is, that $\hat\cM$ is a cyclic operad with respect to $\tau = t^{-1}$ as in \eqref{inverset}.
One directly sees 
for $2 \leq j \leq p$ and $\phi \in \hat\cM(p)$, $\psi \in \hat\cM(q)$ that
\begin{small}
\begin{equation*}
  \begin{split}
& \big(t^{-1} (\phi \circ_j \psi)\big)(g^1, \ldots, g^{p+q-1})
    \\
    =
\ \ &
    (\phi \circ_j \psi)\big(Sg^{p+q-1} \mancino (1, g^1, \ldots, g^{p+q-2})\big)
  \\
  = \ \  &
  \phi\pig(Sg^{p+q-1}_{(4)} \mancino (1, g^1, \ldots, g^{j-3}, g^{j-2}_{(1)}) \,
\psi\big(S(Sg^{p+q-1}_{(3)} g^{j-2}_{(2)}) \mancino Sg^{p+q-1}_{(2)} \mancino (g^{j-1}, \ldots, g^{j+q-3}, g^{j+q-2}_{(1)})\big),
\\
& \qquad
    Sg^{p+q-1}_{(1)} \mancino (g^{j+q-2}_{(2)}, g^{j+q-1}, \ldots, g^{p+q-2})\pig)
\\
    = \ \  &
  \phi\pig(Sg^{p+q-1}_{(2)} \mancino (1, g^1, \ldots, g^{j-3}, g^{j-2}_{(1)}) \,
\psi\big(Sg^{j-2}_{(2)} \mancino (g^{j-1}, \ldots, g^{j+q-3}, g^{j+q-2}_{(1)})\big),
\\
& \qquad
    Sg^{p+q-1}_{(1)} \mancino (g^{j+q-2}_{(2)}, g^{j+q-1}, \ldots, g^{p+q-2})\pig)
    \\
  = \ \  &
  \phi\pig(Sg^{p+q-1}_{(2)} \mancino \pig(1, g^1, \ldots, g^{j-3}, g^{j-2}_{(1)} \,
\psi\big(Sg^{j-2}_{(2)} \mancino (g^{j-1}, \ldots, g^{j+q-3}, g^{j+q-2}_{(1)})\big),
g^{j+q-2}_{(2)}, g^{j+q-1}, \ldots, g^{p+q-2}\pig)\pig)
 \end{split}
  \end{equation*}
\begin{equation*}
  \begin{split}
     =
\ \ &
  (t^{-1} \phi)\pig(g^{1}, \ldots, g^{j-3}, g^{j-2}_{(1)}
  \psi\big(Sg^{j-2}_{(2)} \mancino (g^{j-1}, \ldots, g^{j+q-3}, g^{j+q-2}_{(1)})\big),
  g^{j+q-2}_{(2)},
  g^{j+q-1}, \ldots, g^{p+q-1}\pig) 
\\
=
\ \
& (t^{-1} \phi \circ_{j-1} \psi)(g^1, \ldots, g^{p+q-1}),
   \end{split}
\end{equation*}
\end{small}and likewise for the case $j=1$, which the first line in Eqs.~\eqref{unschoen}. This concludes the proof.  
  \end{proof}

Wrapping up, 
all conditions in Theorem \ref{main} are met for the pair $(\cO, \cM)$ defined in \eqref{jupp}, which therefore results in the following conclusion:

\begin{cor}
  \label{this}
  The cohomology groups $\Ext^\bull_H(k,k)$ for a Hopf algebra $H$ with involutive antipode constitute a BV algebra.
\end{cor}

See, however, the comments made in Remark \ref{final}.

\subsection{$\bfExt$ acting on $\bfTor$ with $\bfCotor$ as a resulting BV algebra}
The preceding subsection can be (linearly) dualised to the case in which the chain complex computing $\Tor$-groups can be seen as a cyclic opposite module over the cochain complex computing $\Ext$-groups, considered as an operad with multiplication. Then, in an essentially analogous way, one can understand how the cyclic dual, being $\Cotor$ in this case as the derived functor of the cotensor product, turns into a BV algebra. This time, 
set
\begin{equation}
  \label{jupp2}
\cO(p) := \Hom(H^{\otimes p}, k),  \qquad \cM(n) := H^{\otimes n},
\end{equation}
where in degree zero $\cO(0) := \cM(0) := k$ as before.
Again, we do not need the explicit composition maps on $\cO$ but only the fact that it is an operad with multiplication, with in this case $(\mu, \mathbb{1}, e) = \big(\gve \circ \mu, \gve, \id_k)$. Specialising Theorem 6.2 in \cite{Kow:GABVSOMOO} to the Hopf algebra case and trivial coefficients, we obtain:
\begin{prop}
For all $1 \leq i \leq n-p+1$ and $0 \leq p \leq n$,
the operations
\begin{equation*}
    \bullet_i \colon \cO(p) \otimes \cM(n) \to \cM(n-p+1), 
\end{equation*}
given by,
for any $\phi \in \cO(p)$ 
and $(h^1, \ldots, h^n) \in \cM(n)$,
\begin{equation*}
  \begin{split}
  & \phi \bullet_i (h^1, \ldots, h^{n})
    := \big(h^1, \ldots, h^{i-1}, \phi(h^i_{(1)}, \ldots, h^{i+p-1}_{(1)})h^i_{(2)} \cdots h^{i+p-1}_{(2)}, h^{i+p}, \ldots, h^{n}\big),
  \end{split}
  \end{equation*}
declared to be zero if $p > n$, along with
\begin{equation*}
1_k \bullet_i (h^1, \ldots, h^{n}) := (h^1 , \ldots,  h^{i-1}, 1_H, h^{i}, \ldots, h^{n})
\end{equation*}
for elements in $\cO(0) = k$ and $1 \leq i \leq n+1$, induce on $\cM$ the structure of a unital opposite $\cO$-module.  
Moreover, defining the extra operation
\begin{equation*}
  \begin{split}
\phi \bullet_0 (h^1, \ldots, h^n) 
    &=
\pig(\phi\big(h^1_{(1)}, \ldots, h^{p-1}_{(1)}, S(h^1_{(2)} \cdots h^{p-1}_{(1)}) \big) h^p, \ldots, h^n\pig),
  \end{split}
  \end{equation*}
declared to be zero if $p > n + 1$, along with the cyclic operator 
\begin{equation}
  \label{cycontor}
  \begin{split}
    t(h^1, \ldots, h^n) &=
\big(S(h^1_{(2)} \cdots h^{n-1}_{(2)}h^n), h^1_{(1)}, \ldots, h^{n-1}_{(1)} \big) 
  \end{split}
  \end{equation}
turns the opposite $\cO$-module $\cM$ into a cyclic one in the sense of Definition \ref{molck}.
\end{prop}

In particular, $\cM$ becomes a $k$-simplicial module, and again a short computation reveals the faces and degeneracies: 
\begin{equation}
  \label{simpontor}
\begin{array}{rcccll}
d_0(h^1, \ldots, h^n) \!\!\!&=\!\!\!\!&  \mu \bullet_0(h^1, \ldots, h^n)
  \!\!\!\!&=\!\!\!& \big(\gve(h^1) h^2, \ldots, h^n\big), &
  \\[1mm]
d_i(h^1, \ldots, h^n) \!\!\!&=\!\!\!\!&
\mu \bullet_i (h^1, \ldots, h^n)
\!\!\!\!&=\!\!\!& (h^1, \ldots, h^ih^{i+1}, \ldots, h^n), &
    \\[1mm]
d_{n}(h^1, \ldots, h^n) \!\!\!&=\!\!\!\!&  \mu \bullet_0 t (h^1, \ldots, h^n)
    \!\!\!\!&=\!\!\!& \big(h^1, \ldots, h^{n-1}\gve(h^n)\big), &
\\[1mm]
s_j (h^1, \ldots, h^n) \!\!\!&=\!\!\!\!&
e \bullet_{j+1} (h^1, \ldots, h^n)
\!\!\!\!&=\!\!\!&
(h^1, \ldots, h^j, 1_H, h^{j+1}, \ldots, h^n), &
\end{array}
\end{equation}
for $1 \leq i \leq n-1$ and $0 \leq j \leq n$, which permits to compute the extra degeneracy:
\begin{equation}
  \label{warmwarmwarm}
  \begin{split}
s_{-1}(h^1, \ldots, h^n) = e \bullet_{0} (h^1, \ldots, h^n)
  &=
  t \big(e \bullet_{n+1} (h^1, \ldots, h^n)\big)
\\
  &=
\big(S(h^1_{(2)} \cdots h^{n}_{(2)}), h^1_{(1)}, \ldots, h^{n}_{(1)} \big). 
\end{split}
  \end{equation}
The cyclic dual $\hat\cM$, considered as a cochain complex by using the prescriptions of Eqs.~\eqref{moulesfrites1}, is hence induced by the following cofaces and codegeneracies:
\begin{equation}
  \label{cosimpontor}
\begin{array}{rcll}
  \gd_0 (h^1, \ldots, h^n)   &=&
  \big(S(h^1_{(2)} \cdots h^{n}_{(2)}), h^1_{(1)}, \ldots, h^{n}_{(1)} \big), 
  \\[1mm]
 \gd_i(h^1, \ldots, h^n) &=&
(h^1, \ldots, h^{i-1}, 1_H, h^{i}, \ldots, h^n), 
  & \quad 1 \leq i \leq n +1,
   \\[1mm]
\gs_0 (h^1, \ldots, h^n) &=&  \big(\gve(h^1)h^2, \ldots, h^n\big), &
\\[1mm]
\gs_j (h^1, \ldots, h^n) &=&  (h^1, \ldots, h^j h^{j+1}, \ldots, h^n),
& \quad 1 \leq j \leq n-1, 
\end{array}
\end{equation}
which describes a cocyclic $k$-module when adding the inverse of \eqref{cycontor} to it, 
\begin{equation}
\label{inversettor}
t^{-1}(h^1, \ldots, h^n) =
\big(h^2_{(1)}, \ldots, h^{n}_{(1)},
S(h^1 h^2_{(2)} \cdots h^{n}_{(2)})\big) 
\end{equation}
in degree $n$.
All that is missing is now the cosimplicial and cocyclic-compatible operadic composition on $\hat\cM$:

\begin{lem}
  The cyclic dual $\hat\cM$
of $\cM = \{ \cM(n)\}_{n \geq 0}$, where $\cM(n) = H^{\otimes n}$, with respect to the cyclic $k$-module structure defined by the Eqs.~\eqref{cycontor} and \eqref{simpontor}, yields a cochain complex that computes $\Cotor^\bull_H(k,k)$. Defining on $\hat\cM$ the partial operadic composition 
  \begin{small}
  \begin{equation}
    \label{verticaltor}
    x \circ_j y =
    \begin{cases}
(g^1, \ldots, g^{q-1}, g^qh^1, h^2, \ldots, h^p)       
      & \mbox{if} \ j = 1,
\\
\begin{array}{r}
  \hspace*{-.2cm}
  \big(h^1, \ldots, h^{j-1} S(g^1_{(2)} \cdots g^q_{(2)}), 
g^1_{(1)}, \ldots, g^q_{(1)}h^j,
 \ldots, h^p \big)
\end{array}
& \mbox{if} \  2 \leq j \leq p,
\end{cases}
  \end{equation}
  \end{small}for $x := (h^1, \ldots, h^p) \in \hat\cM(p)$ and $y := (g^1, \ldots, g^q) \in \hat\cM(q)$,
  one obtains on $\hat\cM$ the structure of a cosimplicial and cocyclic-compatible operad in $\kmod$.
  \end{lem}

\begin{proof}
Quite analogous to the proof of Lemma \ref{ilfautdesjours}, 
restricting the content of \cite[\S3.2]{Kow:ANCCOTCDOE} to Hopf algebras and trivial coefficients, again using a cochain isomorphism obtained from a higher order Hopf-Galois map, results in the fact that the cochain complex obtained from Eqs.~\eqref{cosimpontor} computes $\Cotor^\bull_H(k,k)$.
It also follows by routine computations that the partial compositions \eqref{verticaltor} define an operadic structure
  on $\hat\cM$. Again, that the so-defined operad is 
  cosimplicial-compatible with respect to the cosimplicial operators in Eqs.~\eqref{cosimpontor},
will be carried out in a few cases only. 
As an example, let us check the third line in \eqref{besser2}: to this end, let
$i = j = p+1$. One has
\begin{equation*}
  \begin{split}
& (\gd_{p+1} x) \circ_{p+1} y
    \\
    = \  &
(h^1, \ldots, h^p, 1) \circ_{p+1} (g^1, \ldots, g^q)
    \\
    = \  &
 \big(h^1, \ldots, h^{p-1}, h^{p} S(g^1_{(2)} \cdots g^q_{(2)}), 
 g^1_{(1)}, \ldots, g^q_{(1)}\big)
    \\
    = \  &
 \big(S(g^1_{(2)} \cdots g^q_{(2)}), 
 g^1_{(1)}, \ldots, g^q_{(1)}\big) \circ_1 (h^1, \ldots, h^{p})
   \\
    = \  &
(\gd_0 y) \circ_1 x,
       \end{split}
\end{equation*}
and similarly for all other relations concerning the cofaces $\gd_i$. 
Let us, nevertheless, still explicitly prove one of the relations in 
\eqref{besser3} to exemplify the behaviour of the codegeneracies. In case $1 \leq j \leq p$ and $j-1 < i \leq j+q - 2$, one computes
\begin{equation*}
  \begin{split}
& \gs_i (x \circ_j y)
    \\
    = \  &
\gs_i\big(h^1, \ldots, h^{j-1} S(g^1_{(2)} \cdots g^q_{(2)}), 
g^1_{(1)}, \ldots, g^q_{(1)}h^j,
 \ldots, h^p \big)
    \\
    = \  &
\big(h^1, \ldots, h^{j-1} S(g^1_{(2)} \cdots g^q_{(2)}), 
g^1_{(1)}, \ldots, g^{i-j+1}_{(1)}g^{i-j+2}_{(1)}, \ldots,  g^q_{(1)}h^j,
 \ldots, h^p \big)
    \\
    = \  &
\big(h^1, \ldots, h^{j-1} S(g^1_{(2)} \cdots g^q_{(2)}), 
g^1_{(1)}, \ldots, g^{i-j+1}_{(1)}g^{i-j+2}_{(1)}, \ldots,  g^q_{(1)}h^j,
 \ldots, h^p \big)
    \\
    = \  &
(h^1, \ldots, h^p) \circ_j  (g^1, \ldots, g^{i-j+1} g^{i-j+2}, \ldots,  g^q)
    \\
    = \  &
x \circ_j  (\gs_{i-j+1} y), 
       \end{split}
\end{equation*}
and similarly for the case $i = j+q-2$, 
the only interesting instance being of what happens at the border $i=j-1$. Using $S^2 = \id$, one has:
\begin{equation*}
  \begin{split}
& \gs_{j-1} (x \circ_j y)
    \\
    = \  &
\gs_{j-1}\big(h^1, \ldots, h^{j-1} S(g^1_{(2)} \cdots g^q_{(2)}), 
g^1_{(1)}, \ldots, g^q_{(1)}h^j,
 \ldots, h^p \big)
    \\
    = \  &
\big(h^1, \ldots, h^{j-1} S(g^1_{(2)} \cdots g^q_{(2)})g^1_{(1)}, g^2_{(1)},\ldots,  g^q_{(1)}h^j,
 \ldots, h^p \big)
    \\
    = \  &
\big(h^1, \ldots, h^{j-1}  S(g^2_{(2)} \cdots g^q_{(2)}), \gve(g^1) g^2_{(1)}, \ldots,  g^q_{(1)}h^j,
 \ldots, h^p \big)
    \\
    = \  &
(h^1, \ldots, h^p) \circ_j  \big(\gve(g^1) g^2, \ldots,  g^q\big)
    \\
    = \  &
x \circ_j  (\gs_0 y), 
       \end{split}
\end{equation*}
which finishes the proof of the middle line in relations \eqref{besser3}.
Let us conclude by checking the cyclic operad condition \eqref{unschoen} with respect to $\tau = t^{-1}$ given in \eqref{inversettor}.
For $x, y \in \hat\cM$ as above, one computes
\begin{equation*}
  \begin{split}
& t^{-1} (x \circ_1 y)
    \\
    =
\ \ &
t^{-1}\big((h^1, \ldots, h^p) \circ_1  (g^1, \ldots,  g^q)\big)
  \\
  = \ \  &
t^{-1}(g^1, \ldots,  g^q h^1, \ldots, h^p)
\\
  = \ \  &
\big(g^2_{(1)}, \ldots, g^{q-1}_{(1)}, g^{q}_{(1)}h^1_{(1)}, h^2_{(1)}, \ldots, h^p_{(1)}, 
S(g^1 g^2_{(2)} \cdots g^{q}_{(2)}h^1_{(2)} \cdots  h^p_{(2)})\big) 
    \\
  = \ \  &
\big(g^2_{(1)}, \ldots, g^{q-1}_{(1)}, g^{q}_{(1)}h^1_{(1)}, h^2_{(1)}, \ldots, h^p_{(1)}, 
S(h^1_{(2)} \cdots  h^p_{(2)}) S(g^1 g^2_{(2)} \cdots g^{q}_{(2)})\big) 
    \\
      = \ \  &
    \pig(g^2_{(1)}, \ldots, g^{q-1}_{(1)}, g^{q}_{(1)}S\big(h^2_{(2)} \cdots h^p_{(2)}
    S(h^1 h^2_{(3)} \cdots h^p_{(3)} )_{(2)}\big), h^2_{(1)}, \ldots, h^p_{(1)}, 
\\
    &
  \qquad \qquad \qquad \qquad     \qquad \qquad \qquad \qquad \qquad   
    S(h^1 h^2_{(3)} \cdots h^p_{(3)} )_{(1)} S(g^1 g^2_{(2)} \cdots g^{q}_{(2)})\pig) 
    \\
      = \ \  &
      \big(g^2_{(1)}, \ldots, g^{q}_{(1)}, S(g^1 g^2_{(2)} \cdots g^{q}_{(2)})\big)
      \circ_q \big(h^2_{(1)}, \ldots, h^p_{(1)}, S(h^1 h^2_{(2)} \cdots h^p_{(2)} ) \big) 
    \\
=
\ \
& t^{-1} y \circ_q t^{-1} x,
   \end{split}
\end{equation*}
and likewise for the cases $2 \leq j \leq q$. Hence, Eqs.~\eqref{unschoen} are true, which concludes the proof.  
\end{proof}

As in the preceding example, 
again all conditions in Theorem \ref{main} are met for the pair $(\cO, \cM)$ defined in \eqref{jupp2}, and therefore we can state:

\begin{cor}
  \label{that}
  The cohomology groups $\Cotor^\bull_H(k,k)$ for a Hopf algebra $H$ with involutive antipode constitute a BV algebra.
\end{cor}

\begin{rem}
\label{final}
  The existence of BV structures as stated in Corollaries \ref{this} \& \ref{that} has been observed before:
alternatively, the BV structure on $\Cotor$-groups can be obtained by defining the structure of a cyclic operad {\em with multiplication} on the relevant cochain complex. This has been dealt with first in \cite{Men:BVAACCOHA}, and can be constructed in an essentially analogous way for $\Ext$-groups, see \cite{Kow:WEIABVA}. 

As hinted at at the beginning of this example section, the two preceding subsections
can be generalised not only to Hopf algebroids but also more general coefficients than the ground ring $k$ can be considered, see \cite{Kow:ANCCOTCDOE} for a detailed account.
The Gerstenhaber algebra structure, {\em e.g.}, on $\Ext$-groups essentially asks for (braided commutative resp.\ cocommutative) monoids resp.\ comonoids in the monoidal centre of left $H$-modules as coefficients, see \cite{FioKow:BAPFCIEC} for quite a general approach in extension categories in the spirit of Schwede \cite{Schw:AESIOTLBIHC}, whereas for $\Cotor$-groups one deals, in a similar way, with the monoidal category of left $H$-comodules.
A cyclic operator $\tau$ on the complexes computing $\Ext$ (resp.~$\Cotor$) with values in more general coefficients, however, requires these to be objects not in a monoidal centre but rather in a {\em bimodule category centre}. This is to say, a centre construction when considering the categories of left $H$-modules resp.\ $H$-comodules as bimodule categories over its respective opposite by means of a left and right (sort of) adjoint action induced by the respective internal left and right homs: see \cite{Kow:CTACC} for an approach in the framework of Hopf algebroids. How these two properties fit together to produce BV structures on cohomology with general coefficients, {\em i.e.}, objects being simultaneously {\em monoidally central} and {\em bimodule central} is an open question currently under examination. 
  \end{rem}

\providecommand{\bysame}{\leavevmode\hbox to3em{\hrulefill}\thinspace}
\providecommand{\MR}{\relax\ifhmode\unskip\space\fi M`R }
\providecommand{\MRhref}[2]{%
  \href{http://www.ams.org/mathscinet-getitem?mr=#1}{#2}}
\providecommand{\href}[2]{#2}

\end{document}